\documentclass[11pt]{article}  
\usepackage{amsmath}
\usepackage{amssymb}
\usepackage{theorem}
\usepackage{euscript}
\usepackage{pstricks}
\title{\sffamily Dualization of Signal Recovery 
Problems\footnote{Contact author: 
P. L. Combettes, {\ttfamily plc@math.jussieu.fr},
phone: +33 1 4427 6319, fax: +33 1 4427 7200. The work of 
P. L. Combettes was supported the Agence Nationale de la Recherche 
under grant ANR-08-BLAN-0294-02. The work of 
D\hspace{-1.6ex}\raise 0.4ex\hbox{-}\hspace{.8ex}inh D\~ung and
B$\grave{\text{\u{a}}}$ng C\^ong V\~u was supported by the
Vietnam National Foundation for Science and Technology Development.
}}
\author{Patrick L. Combettes,$^1$ 
D\hspace{-1.6ex}\raise 0.4ex\hbox{-}\hspace{.8ex}inh D\~ung,$^2$ 
and B$\grave{\text{\u{a}}}$ng C\^ong V\~u$^3$\\[5mm]
\small$\!^1$UPMC Universit\'e Paris 06\\
\small Laboratoire Jacques-Louis Lions -- UMR 7598\\
\small 75005 Paris, France\\
\small {\ttfamily plc@math.jussieu.fr}\\[4mm]
\small $\!^2$Vietnam National University\\
\small Information Technology Institute\\
\small Hanoi, Vietnam\\
\small {\ttfamily dinhdung@vnu.edu.vn}\\[4mm]
\small $\!^3$Vietnam National University\\
\small Department of Mathematics -- Mechanics -- Informatics\\ 
\small Hanoi, Vietnam\\
\small {\ttfamily bangvc@vnu.edu.vn}
}
\date{~}
\newcommand{\Frac}[2]{\displaystyle{\frac{#1}{#2}}} 
\newcommand{\scal}[2]{{\left\langle{{#1}\mid{#2}}\right\rangle}}

\newcommand{\menge}[2]{\big\{{#1}~\big |~{#2}\big\}} 
\newcommand{\Menge}[2]{\bigg\{{#1}~\bigg |~{#2}\bigg\}} 

\newcommand{\BL}{\ensuremath{\EuScript B}\,}
\newcommand{\RX}{\ensuremath{\left]-\infty,+\infty\right]}}

\newcommand{\soft}[1]{\ensuremath{{\operatorname{soft}}_{{#1}}\,}}
\newcommand{\HH}{\ensuremath{{\mathcal H}}}

\newcommand{\GG}{\ensuremath{{\mathcal G}}}

\newcommand{\emp}{\ensuremath{{\varnothing}}}
\newcommand{\prox}{\ensuremath{\operatorname{prox}}}

\newcommand{\Id}{\ensuremath{\operatorname{Id}}\,}

\newcommand{\cone}{\ensuremath{\operatorname{cone}}}
\newcommand{\RR}{\ensuremath{\mathbb{R}}}

\newcommand{\RP}{\ensuremath{\left[0,+\infty\right[}}

\newcommand{\RPP}{\ensuremath{\left]0,+\infty\right[}}
\newcommand{\NN}{\ensuremath{\mathbb N}}
\newcommand{\KK}{\ensuremath{\mathbb K}}
\newcommand{\LL}{\ensuremath{\mathbb L}}

\newcommand{\weakly}{\ensuremath{\:\rightharpoonup\:}}
\newcommand{\exi}{\ensuremath{\exists\,}}
\newcommand{\pinf}{\ensuremath{{+\infty}}}

\newcommand{\dom}{\ensuremath{\operatorname{dom}}}
\newcommand{\sri}{\ensuremath{\operatorname{sri}}}
\newcommand{\core}{\ensuremath{\operatorname{core}}}
\newcommand{\reli}{\ensuremath{\operatorname{ri}}}
\newcommand{\spc}{\ensuremath{\overline{\operatorname{span}}\,}}
\newcommand{\spa}{\ensuremath{{\operatorname{span}}\,}}

\newcommand{\sign}{\ensuremath{\operatorname{sign}}}
\newcommand{\inte}{\ensuremath{\operatorname{int}}}
\newcommand{\rint}{\ensuremath{\operatorname{ri}}}
   
\newtheorem{theorem}{Theorem}[section]
\newtheorem{lemma}[theorem]{Lemma}

\theoremstyle{plain}{\theorembodyfont{\rmfamily}%
}
\newtheorem{proposition}[theorem]{Proposition}
\theoremstyle{plain}{\theorembodyfont{\rmfamily}%
}
\theoremstyle{plain}{\theorembodyfont{\rmfamily}%
\newtheorem{algorithm}[theorem]{Algorithm}}
\theoremstyle{plain}{\theorembodyfont{\rmfamily}%
\newtheorem{example}[theorem]{Example}}
\theoremstyle{plain}{\theorembodyfont{\rmfamily}%
\newtheorem{remark}[theorem]{Remark}}
\theoremstyle{plain}{\theorembodyfont{\rmfamily}%
}
\theoremstyle{plain}{\theorembodyfont{\rmfamily}%
\newtheorem{problem}[theorem]{Problem}}

\numberwithin{equation}{section}
\begin{document}
\maketitle

\begin{abstract}
In convex optimization, duality theory can sometimes lead to simpler 
solution methods than those resulting from direct primal analysis. 
In this paper, this principle is applied to a class of composite 
variational problems arising in particular in signal recovery. 
These problems are not easily amenable to solution by 
current methods but they feature Fenchel-Moreau-Rockafellar dual
problems that can be solved by forward-backward splitting.
The proposed algorithm produces simultaneously a sequence converging 
weakly to a dual solution, and a sequence converging strongly to 
the primal solution. Our framework is shown to capture and extend 
several existing duality-based signal recovery methods and to 
be applicable to a variety of new problems beyond their scope.
\end{abstract}



\newpage
\section{Introduction}
Over the years, several structured frameworks have been proposed to 
unify the analysis and the numerical solution methods of classes of 
signal (including image) recovery problems. 
An early contribution was made by Youla in 
1978 \cite{Youl78}. He showed that several signal recovery problems, 
including those of \cite{Gerc74,Papo75}, shared a simple common 
geometrical structure and could be reduced to the following 
formulation in a Hilbert space $\HH$ with scalar product
$\scal{\cdot}{\cdot}$ and associated norm $\|\cdot\|$: 
find the signal in a closed vector 
subspace $C$ which admits a known projection $r$ onto a closed vector 
subspace $V$, and which is at minimum distance from some reference 
signal $z$. This amounts to solving the variational problem
\begin{equation}
\label{e:hanoi2009-01-07}
\underset{\substack{x\in C\\P_Vx=r}}{\mathrm{minimize}}\;\;
\frac12\|x-z\|^2,
\end{equation}
where $P_V$ denotes the projector onto $V$. Abstract Hilbert space signal 
recovery problems have also been investigated by other authors. For
instance, in 1965, Levi \cite{Levi65} considered the problem of finding 
the minimum energy band-limited signal fitting $N$ linear measurements. 
In the Hilbert space $\HH=L^2(\RR)$, the underlying variational problem
is to 
\begin{equation}
\label{e:1965}
\underset{\substack{x\in C\\ 
\scal{x}{s_1}=\rho_1\\
~~~~\vdots\\
\scal{x}{s_N}=\rho_N}}{\text{minimize}}\;\;
\frac12\|x\|^2,
\end{equation}
where $C$ is the subspace of band-limited signals, 
$(s_i)_{1\leq i\leq N}\in\HH^N$ are the measurement signals, and 
$(\rho_i)_{1\leq i\leq N}\in\RR^N$ are the measurements. In \cite{Pott93},
Potter and Arun observed that, for a general closed convex set $C$, 
the formulation \eqref{e:1965} models a variety of problems, ranging 
from spectral estimation \cite{Bent88,Stei85} and tomography
\cite{Medo87}, to other inverse problems \cite{Bert85}. In addition,
they employed an elegant duality framework to solve it, which led to 
the following result. 

\begin{proposition}{\rm\cite[Theorems~1~and~3]{Pott93}}
\label{p:lee93}
Set $r=(\rho_i)_{1\leq i\leq N}$ and 
$L\colon\HH\to\RR^N\colon x\mapsto(\scal{x}{s_i})_{1\leq i\leq N}$, 
and let $\gamma\in\left]0,2\right[$. Suppose that 
$\sum_{i=1}^N\|s_i\|^2\leq 1$ and that $r$ lies in the relative
interior of $L(C)$. Set 
\begin{equation}
\label{e:main91}
w_0\in\RR^N\quad\text{and}\quad(\forall n\in\NN)\quad
w_{n+1}=w_n+\gamma\big(r-LP_CL^*w_n\big),
\end{equation}
where $L^*\colon\RR^N\to\HH\colon(\nu_i)_{1\leq i\leq N}\mapsto%
\sum_{i=1}^N\nu_is_i$ is the adjoint of $L$.
Then $(w_n)_{n\in\NN}$ converges to a point $w$ such that $LP_CL^*w=r$
and $P_CL^*w$ is the solution to \eqref{e:1965}.
\end{proposition}

Duality theory plays a central role in convex optimization
\cite{Ekel99,More66,Rock74,Zali02} and it has been used, in 
various forms and with different objectives, in several places in 
signal recovery, e.g., 
\cite{Bent88,Borw96,Cham04,Chan99,Smms05,Dest07,Fadi09,Hint06,%
Iuse93,Leah86,Weis09}; let us add that, since the completion of 
the present paper \cite{Arxi09}, other aspects of duality in imaging
have been investigated in \cite{Borw10}. 
For our purposes, the most suitable type of duality is the so-called 
Fenchel-Moreau-Rockafellar duality, which associates to a composite
minimization problem a ``dual'' minimization problem involving the 
conjugates of the functions and the adjoint of the linear operator
acting in the primal problem. In general, the dual problem sheds a new 
light on the properties of the primal problem and enriches its 
analysis. Moreover, in certain specific situations, it is actually 
possible to solve the dual problem and to recover a solution to 
the primal problem from any dual solution.
Such a scenario underlies Proposition~\ref{p:lee93}: the primal problem
\eqref{e:1965} is difficult to solve but, if $C$ is simple enough, 
the dual problem can be solved efficiently and, furthermore, a primal 
solution can be recovered explicitly.
This principle is also explicitly or implicitly present in other signal
recovery problems. For instance, the variational denoising problem 
\begin{equation}
\label{e:2005}
\underset{x\in\HH}{\mathrm{minimize}}\;\;
g(Lx)+\frac12\|x-z\|^2,
\end{equation}
where $z$ is a noisy observation of an ideal signal, $L$ is a bounded 
linear operator from $\HH$ to some Hilbert space $\GG$, and 
$g\colon\GG\to\RX$ is a proper lower semicontinuous convex function,
can often be approached efficiently using duality arguments 
\cite{Smms05}. A popular development in this direction is the
total variation denoising algorithm proposed in \cite{Cham04} and
refined in \cite{Cham05}. 

The objective of the present paper is to devise a duality 
framework that captures problems such as \eqref{e:hanoi2009-01-07}, 
\eqref{e:1965}, and \eqref{e:2005} and leads to improved algorithms 
and convergence results, in an effort to standardize the use of 
duality techniques in signal recovery and extend their range of 
potential applications. More specifically, we focus on a class of 
convex variational problems which satisfy the following.
\begin{itemize}
\item[(a)]
They cover the above minimization problems.
\item[(b)]
They are not easy to solve directly, but they admit a 
Fenchel-Moreau-Rockafellar dual which can be solved reliably 
in the sense that an implementable 
algorithm is available with proven weak or strong convergence to a 
solution of the whole sequence of iterates it generates. In some cases, 
asymptotic properties of a primal sequence are also desirable.
Here ``implementable'' is taken in the classical sense of
\cite{Pola71}: the algorithm does not involve subprograms (e.g., 
``oracles'' or ``black-boxes'') which are not guaranteed to 
converge in a finite number of steps.
\item[(c)]
They allow for the construction of a primal solution from any dual 
solution.
\end{itemize}
A problem formulation which complies with these requirements is 
the following, where we denote by $\sri C$ the strong relative interior 
of a convex set $C$ (see \eqref{e:sri} and Remark~\ref{r:2008}).

\begin{problem}[primal problem]
\label{prob:1}
Let $\HH$ and $\GG$ be real Hilbert spaces, let $z\in\HH$, let $r\in\GG$, 
let $f\colon\HH\to\RX$ and $g\colon\GG\to\RX$ be lower semicontinuous
convex functions, and let $L\colon\HH\to\GG$ be a nonzero linear 
bounded operator such that the qualification condition 
\begin{equation}
\label{e:1938}
r\in\sri\big(L(\dom f)-\dom g\big)
\end{equation}
holds. The problem is to 
\begin{equation}
\label{e:prob1}
\underset{x\in\HH}{\mathrm{minimize}}\;\;
f(x)+g(Lx-r)+\frac12\|x-z\|^2.
\end{equation}
\end{problem}

In connection with (a), it is clear that \eqref{e:prob1} covers 
\eqref{e:2005} for $f=0$.  Moreover, if we let $f$ and $g$ be the 
indicator functions (see \eqref{e:iota}) of closed convex sets 
$C\subset\HH$ and $D\subset\GG$, respectively, then \eqref{e:prob1} 
reduces to the best approximation problem
\begin{equation}
\label{e:porquerolles2009-06-12}
\underset{\substack{x\in C\\Lx-r\in D}}{\mathrm{minimize}}\;\;
\frac12\|x-z\|^2,
\end{equation}
which captures both \eqref{e:hanoi2009-01-07} and \eqref{e:1965} in the
case when $C$ is a closed vector subspace and $D=\{0\}$. Indeed, 
\eqref{e:hanoi2009-01-07} corresponds to $\GG=\HH$ and $L=P_V$, while 
\eqref{e:1965} corresponds to $\GG=\RR^N$, 
$L\colon\HH\to\RR^N\colon x\mapsto(\scal{x}{s_i})_{1\leq i\leq N}$,
$r=(\rho_i)_{1\leq i\leq N}$, and $z=0$. As will be seen in 
Section~\ref{sec:3}, Problem~\ref{prob:1} models a broad range of 
additional signal recovery problems. 

In connection with (b), it is natural to ask whether the minimization
problem \eqref{e:prob1} can be solved reliably by existing algorithms. 
Let us set
\begin{equation}
\label{e:24juin2009}
h\colon\HH\to\RX\colon x\mapsto f(x)+g(Lx-r). 
\end{equation}
Then it follows from
\eqref{e:1938} that $h$ is a proper lower semicontinuous convex
function. Hence its proximity operator $\prox_h$, which maps 
each $y\in\HH$ to the unique minimizer of the function 
$x\mapsto h(x)+\|y-x\|^2/2$, is well defined 
(see Section~\ref{sec:prox1}). Accordingly, Problem~\ref{prob:1} 
possesses a unique solution, which can be concisely written as 
\begin{equation}
\label{e:rio-mai2009}
x=\prox_h z.
\end{equation}
Since no-closed form expression exists for the proximity operator of 
composite functions such as $h$, one can contemplate the use of 
splitting strategies to construct $\prox_h z$ since \eqref{e:prob1} 
is of the form
\begin{equation}
\label{e:prob0}
\underset{x\in\HH}{\text{minimize}}\;\;f_1(x)+f_2(x),
\end{equation}
where 
\begin{equation}
\label{e:porquerolles2009-06-11}
f_1\colon x\mapsto f(x)+\frac12\|x-z\|^2\quad\text{and}\quad
f_2\colon x\mapsto g(Lx-r) 
\end{equation}
are lower semicontinuous convex functions from $\HH$ to $\RX$. To tackle 
\eqref{e:prob0}, a first splitting framework is that described in 
\cite{Smms05}, which requires the additional assumption that $f_2$ be
Lipschitz-differentiable on $\HH$ (see also 
\cite{Biou07,Lore08,Caij08,Caij09,Chau07,Siop07,Daub07,Forn07} 
for recent work within this setting). In this case, \eqref{e:prob0}
can be solved by the proximal forward-backward algorithm, which is 
governed by the updating rule
\begin{equation}
\label{e:main2004}
\left\lfloor
\begin{array}{lll}
x_{n+\frac12}&\!\!\!=\!\!\!&\nabla f_2(x_n)+a_{2,n}\\
x_{n+1}      &\!\!\!=\!\!\!&x_n+\lambda_n\Big(\prox_{\gamma_n f_1}
\big(x_n-\gamma_nx_{n+\frac12}\big)+a_{1,n}-x_n\Big),
\end{array}
\right.\\
\end{equation}
where $\lambda_n>0$ and $\gamma_n>0$, and where $a_{1,n}$ and 
$a_{2,n}$ model respectively tolerances in the approximate 
implementation of the proximity operator of $f_1$ and the 
gradient of $f_2$. Precise
convergence results for the iterates $(x_n)_{n\in\NN}$ can be found
in Theorem~\ref{t:0}. Let us add that there exist variants of this
splitting method, which do not guarantee convergence of the iterates
but do provide an optimal (in the sense of \cite{Nemi83}) 
$O(1/n^2)$ rate of convergence of the objective values 
\cite{Beck09}. A 
limitation of this first framework is that it imposes that $g$ be
Lipschitz-differentiable and therefore excludes key problems such as
\eqref{e:porquerolles2009-06-12}. An alternative framework, which 
does not demand any smoothness assumption in \eqref{e:prob0}, is 
investigated in \cite{Jsts07}. It employs the Douglas-Rachford 
splitting algorithm, which revolves around the updating rule 
\begin{equation}
\label{e:7ans}
\left\lfloor
\begin{array}{lll}
x_{n+\frac12}&\!\!\!=\!\!\!&\prox_{\gamma f_2}x_n+a_{2,n}\\
x_{n+1}      &\!\!\!=\!\!\!&x_n+
\lambda_n\Big(\prox_{\gamma f_1}\big(2x_{n+\frac12}-x_n\big)
+a_{1,n}-x_{n+\frac12}\Big),
\end{array}
\right.\\
\end{equation}
where $\lambda_n>0$ and $\gamma>0$, and where $a_{1,n}$ and 
$a_{2,n}$ model tolerances in the approximate implementation of 
the proximity operators of $f_1$ and $f_2$, respectively 
(see \cite[Theorem~20]{Jsts07} for precise convergence results 
and \cite{Chau09} for further applications).
However, this approach requires that the proximity operator of
the composite function $f_2$ in \eqref{e:porquerolles2009-06-11} 
be computable to within some quantifiable error. Unfortunately, 
this is not possible in general, as explicit expressions 
of $\prox_{g\circ L}$ in terms of $\prox_g$ require stringent assumptions,
for instance $L\circ L^*=\kappa\Id$ for some $\kappa>0$ (see
Example~\ref{ex:1931}), which does not hold in the case of
\eqref{e:1965} and many other important problems. A third framework
that appears to be relevant is that of \cite{Pjo208}, which is 
tailored for problems of the form 
\begin{equation}
\label{e:prob10}
\underset{x\in\HH}{\text{minimize}}\;\;h_1(x)+h_2(x)+\frac12\|x-z\|^2,
\end{equation}
where $h_1$ and $h_2$ are lower semicontinuous convex functions from 
$\HH$ to $\RX$ such that $\dom h_1\cap \dom h_2\neq\emp$. This 
formulation coincides with our setting for $h_1=f$ and 
$h_2\colon x\mapsto g(Lx-r)$. The Dykstra-like 
algorithm devised in \cite{Pjo208} to solve
\eqref{e:prob10} is governed by the iteration 
\begin{equation}
\label{e:6ans}
\begin{array}{l}
\operatorname{Initialization}\\
\left\lfloor
\begin{array}{lll}
y_0=z\\
q_0=0\\
p_0=0
\end{array}
\right.\\[6mm]
\operatorname{For}\;n=0,1,\ldots\\
\left\lfloor
\begin{array}{lll}
x_n&=&\prox_{h_2}(y_n+q_n)\\
q_{n+1}&=&y_n+q_n-x_n\\
y_{n+1}&=&\prox_{h_1}(x_n+p_n)\\
p_{n+1}&=&x_n+p_n-y_{n+1}
\end{array}
\right.\\
\end{array}
\end{equation}
and therefore requires that the proximity operators of $h_1$ and $h_2$ 
be computable explicitly. As just discussed, this is seldom possible 
in the case of the composite function $h_2$. To sum up, existing
splitting techniques do not offer satisfactory options to solve
Problem~\ref{prob:1} and alternative routes must be explored. The
cornerstone of our paper is that, by contrast, Problem~\ref{prob:1} 
can be solved reliably via Fenchel-Moreau-Rockafellar duality so long 
as the operators $\prox_f$ and $\prox_g$ can be evaluated to within 
some quantifiable error, which will be shown to be possible in 
a wide variety of problems. 

The paper is organized as follows. In Section~\ref{sec:2-} we provide
the convex analytical background required in subsequent sections and, in
particular, we review proximity operators. In Section~\ref{sec:2}, we
show that Problem~\ref{prob:1} satisfies properties (b) and (c). We then 
derive the Fenchel-Moreau-Rockafellar dual of Problem~\ref{prob:1}
and then show that it is amenable to solution by forward-backward
splitting. The resulting primal-dual algorithm involves the functions 
$f$ and $g$, as well as the operator $L$, separately and therefore 
achieves full splitting of the constituents of the primal problem. 
We show that the primal sequence produced by the algorithm converges 
strongly to the solution to Problem~\ref{prob:1}, and that the dual
sequence converges weakly to a solution to the dual problem. 
Finally, in Section~\ref{sec:3}, we highlight applications of the 
proposed duality framework to best approximation problems,
denoising problems using dictionaries, and recovery problems involving
support functions. In particular, we extend and provide
formal convergence results for the total variation denoising
algorithm proposed in \cite{Cham05}. Although signal
recovery applications are emphasized in the present paper, the proposed
duality framework is applicable to any variational problem conforming to
the format described in Problem~\ref{prob:1}.

\section{Convex-analytical tools}
\label{sec:2-}

\subsection{General notation}
Throughout the paper, $\HH$ and $\GG$ are real Hilbert spaces, 
and $\BL(\HH,\GG)$ is the space of bounded 
linear operators from $\HH$ to $\GG$. The identity operator is denoted
by $\Id\!$, the adjoint of an operator $T\in\BL(\HH,\GG)$ by $T^*$,
the scalar products of both $\HH$ and $\GG$ by 
$\scal{\cdot}{\cdot}$ and the associated norms by $\|\cdot\|$. 
Moreover, $\weakly$ and $\to$ denote respectively weak and strong
convergence. Finally, we denote by $\Gamma_0(\HH)$ the class of lower 
semicontinuous convex functions $\varphi\colon\HH\to\RX$ which 
are proper in the sense that 
$\dom\varphi=\menge{x\in\HH}{\varphi(x)<\pinf}\neq\emp$. 

\subsection{Convex sets and functions}
We provide some background on convex analysis; for a detailed account,
see \cite{Zali02} and, for finite-dimensional spaces, \cite{Rock70}.

Let $C$ be a nonempty convex subset of $\HH$. The indicator function of 
$C$ is
\begin{equation}
\label{e:iota}
\iota_C\colon x\mapsto
\begin{cases}
0,&\text{if}\;\;x\in C;\\
\pinf,&\text{if}\;\;x\notin C,
\end{cases}
\end{equation}
the distance function of $C$ is 
\begin{equation}
\label{e:d_C}
d_C\colon\HH\to\RP\colon x\mapsto\inf_{y\in C}\|x-y\|,
\end{equation}
the support function of $C$ is
\begin{equation}
\label{e:support}
\sigma_C\colon\HH\to\RX\colon{u}\mapsto\sup_{x\in C}\scal{x}{u},
\end{equation}
and the conical hull of $C$ is  
\begin{equation}
\label{e:cone}
\cone C=\bigcup_{\lambda>0}\menge{\lambda x}{x\in C}.
\end{equation}
If $C$ is also closed, the projection of a point $x$ in $\HH$ onto 
$C$ is the unique point $P_Cx$ in $C$ such that $\|x-P_Cx\|=d_C(x)$. 
We denote by $\inte C$ the interior of $C$, by
$\spa C$ the span of $C$, and by $\spc C$ the closure
of $\spa C$. The core of $C$ is 
$\core C=\menge{x\in C}{\cone (C-x)=\HH}$,
the strong relative interior of $C$ is 
\begin{equation}
\label{e:sri}
\sri C=\menge{x\in C}{\cone (C-x)=\spc (C-x)},
\end{equation}
and the relative interior of $C$ is 
$\reli C=\menge{x\in C}{\cone (C-x)=\spa (C-x)}$.
We have 
\begin{equation}
\label{e:palawan-mai08-1}
\inte C\subset\core C\subset\sri C\subset\reli C\subset C.
\end{equation}
The strong relative interior is therefore an extension of the notion
of an interior. This extension is particularly important in convex
analysis as many useful sets have empty interior infinite-dimensional 
spaces. 

\begin{remark}
\label{r:2008}
The qualification condition \eqref{e:1938} in Problem~\ref{prob:1} is 
rather mild. In view of \eqref{e:palawan-mai08-1}, it is satisfied in 
particular when $r$ belongs to the core and, a fortiori, to the 
interior of $L(\dom f)-\dom g$; the latter is for instance satisfied 
when $L(\dom f)\cap(r+\inte\dom g)\neq\emp$. 
If $f$ and $g$ are proper, then \eqref{e:1938} is also satisfied when 
$L(\dom f)-\dom g=\HH$ and, a fortiori, when $f$ is 
finite-valued and $L$ is surjective, or when $g$ is finite-valued.
If $\GG$ is finite-dimensional, then \eqref{e:1938} reduces to 
\cite[Section~6]{Rock70} 
\begin{equation}
\label{e:2009-06-30}
r\in\rint\big(L(\dom f)-\dom g\big)=(\rint L(\dom f))-\rint \dom g,
\end{equation}
i.e., $(\rint L(\dom f))\cap(r+\rint\dom g)\neq\emp$.
\end{remark}

Let $\varphi\in\Gamma_0(\HH)$. The conjugate of $\varphi$
is the function $\varphi^*\in\Gamma_0(\HH)$ defined by 
\begin{equation}
\label{e:conjugate}
(\forall u\in\HH)\quad\varphi^*(u)=\sup_{x\in\HH}\scal{x}{u}-\varphi(x). 
\end{equation}
The Fenchel-Moreau theorem states that $\varphi^{**}=\varphi$.
The subdifferential of $\varphi$ is the set-valued operator
\begin{equation}
\label{e:subdiff}
\partial\varphi\colon\HH\to 2^{\HH}\colon x\mapsto
\menge{u\in\HH}{(\forall y\in\HH)\;\:
\scal{y-x}{u}+\varphi(x)\leq\varphi(y)}.
\end{equation}
We have
\begin{equation}
\label{e:terminal2E}
(\forall (x,u)\in\HH\times\HH)\quad u\in\partial\varphi(x)\quad
\Leftrightarrow\quad x\in\partial\varphi^*(u).
\end{equation}
Moreover, if $\varphi$ is G\^ateaux differentiable at $x$, then 
\begin{equation}
\label{e:moreau23}
\partial\varphi(x)=\{\nabla\varphi(x)\}. 
\end{equation}
Fermat's rule states that
\begin{equation}
\label{e:fermat}
(\forall x\in\HH)\quad x\in\operatorname{Argmin}\varphi=
\menge{x\in\dom\varphi}{(\forall y\in\HH)\;\;\varphi(x)\leq\varphi(y)}
\quad\Leftrightarrow\quad 0\in\partial\varphi(x).
\end{equation}
If $\operatorname{Argmin}\varphi$ is a singleton, we denote by 
$\operatorname{argmin}_{y\in\HH}\varphi(y)$ the unique minimizer 
of $\varphi$. 

\begin{lemma} {\rm\cite[Theorem~2.8.3]{Zali02}}
\label{l:4}
Let $\varphi\in\Gamma_0(\HH)$, let $\psi\in\Gamma_0(\GG)$, and 
let $M\in\BL(\HH,\GG)$ be such that $0\in\sri(M(\dom\varphi)-\dom\psi)$.
Then $\partial(\varphi+\psi\circ M)=\partial\varphi+
M^*\circ(\partial\psi)\circ M$.
\end{lemma}

\subsection{Moreau envelopes and proximity operators}
\label{sec:prox1}
Essential to this paper is the notion of a proximity operator, which 
is due to Moreau \cite{Mor62b} (see \cite{Smms05,More65} for detailed 
accounts and Section~\ref{sec:22} for closed-form examples).
The Moreau envelope of $\varphi$ is the continuous convex function
\begin{equation}
\label{e:moreau1}
\widetilde{\varphi}\colon\HH\to\RR\colon x\mapsto
\min_{y\in\HH}\:\varphi(y)+\frac12\|x-y\|^2.
\end{equation}
For every $x\in\HH$, the function $y\mapsto\varphi(y)+\|x-y\|^2/2$ 
admits a unique minimizer, which is denoted by $\prox_\varphi x$. The 
proximity operator of $\varphi$ is defined by
\begin{equation}
\label{e:moreau2}
\prox_{\varphi}\colon\HH\to\HH\colon x\mapsto
\underset{y\in\HH}{\operatorname{argmin}}\:\varphi(y)+\frac12\|x-y\|^2
\end{equation}
and characterized by 
\begin{equation}
\label{e:prox1}
(\forall (x,p)\in\HH\times\HH)\quad p=\prox_\varphi
x\quad\Leftrightarrow\quad x-p\in\partial\varphi(p).
\end{equation}

\begin{lemma}{\rm\cite{More65}}
\label{l:6}
Let $\varphi\in\Gamma_0(\HH)$. Then the following hold.
\begin{enumerate}
\item
\label{l:6ii-}
$(\forall x\in\HH)(\forall y\in\HH)$ 
$\|\prox_\varphi x-\prox_\varphi y\|^2\leq
\scal{x-y}{\prox_\varphi x-\prox_\varphi y}$.
\item
\label{l:6ii}
$(\forall x\in\HH)(\forall y\in\HH)$ 
$\|\prox_\varphi x-\prox_\varphi y\|\leq\|x-y\|$.
\item
\label{l:6i}
$\widetilde{\varphi}+\widetilde{\varphi^*}=\|\cdot\|^2/2$.
\item
\label{l:6iii}
$\widetilde{\varphi^*}$ is Fr\'echet differentiable and 
$\nabla\widetilde{\varphi^*}=\prox_{\varphi}=\Id-\prox_{\varphi^*}$.
\end{enumerate}
\end{lemma}

The identity $\prox_{\varphi}=\Id-\prox_{\varphi^*}$ can be stated
in a slightly extended context.
\begin{lemma}{\rm\cite[Lemma~2.10]{Smms05}}
\label{l:7}
Let $\varphi\in\Gamma_0(\HH)$, let $x\in\HH$, and let $\gamma\in\RPP$. 
Then $x=\prox_{\gamma\varphi}x+
\gamma\prox_{\gamma^{-1}\varphi^*}(\gamma^{-1}x)$.
\end{lemma}

The following fact will also be required.
\begin{lemma}
\label{l:2}
Let $\psi\in\Gamma_0(\HH)$, let $w\in\HH$, and set
$\varphi\colon x\mapsto\psi(x)+\|x-w\|^2/2$. Then
$\varphi^*\colon u\mapsto\widetilde{\psi^*}(u+w)-\|w\|^2/2$. 
\end{lemma}
\begin{proof}
Let $u\in\HH$. It follows from \eqref{e:conjugate} and
Lemma~\ref{l:6}\ref{l:6i} that
\begin{align}
\varphi^*(u)
&=-\inf_{x\in\HH}\psi(x)+\frac12\|x-w\|^2-\scal{x}{u}\nonumber\\ 
&=\frac12\|u\|^2+\scal{w}{u}-
\inf_{x\in\HH}\psi(x)+\frac12\|x-(w+u)\|^2\nonumber\\ 
&=\frac12\|u+w\|^2-\frac12\|w\|^2-\widetilde{\psi}(u+w)\nonumber\\
&=\widetilde{\psi^*}(u+w)-\frac12\|w\|^2, 
\end{align}
which yields the desired identity.
\end{proof}

\subsection{Examples of proximity operators}
\label{sec:22}

To solve Problem~\ref{prob:1}, our algorithm will use (approximate) 
evaluations of the proximity operators of the functions $f$ and $g^*$ 
(or, equivalently, of $g$ by Lemma~\ref{l:6}\ref{l:6iii}).
In this section, we supply examples of proximity operators which admit
closed-form expressions.

\begin{example}
\label{ex:1938}
Let $C$ be a nonempty closed convex subset of $\HH$. Then 
the following hold.
\begin{enumerate}
\item
\label{ex:1938i}
Set $\varphi=\iota_C$. Then $\prox_\varphi=P_C$
{\rm\cite[Example~3.d]{More65}}.
\item
\label{ex:1938ii}
Set $\varphi=\sigma_C$. Then $\prox_\varphi=\Id-P_C$
{\rm\cite[Example~2.17]{Smms05}}.
\item
Set $\varphi=d_C^2/(2\alpha)$. Then $(\forall x\in\HH)$ 
$\prox_\varphi x=x+(1+\alpha)^{-1}(P_Cx-x)$ 
{\rm\cite[Example~2.14]{Smms05}}.
\item
Set $\varphi=(\|\cdot\|^2-d_C^2)/(2\alpha)$. Then $(\forall x\in\HH)$ 
$\prox_\varphi x=x-\alpha^{-1}P_C(\alpha(\alpha+1)^{-1}x)$ 
{\rm\cite[Lemma~2.7]{Smms05}}.
\end{enumerate}
\end{example}

\begin{example}{\rm\cite[Lemma~2.7]{Smms05}}
\label{ex:1933}
Let $\psi\in\Gamma_0(\HH)$ and set
$\varphi=\|\cdot\|^2/2-\widetilde{\psi}$. Then $\varphi\in\Gamma_0(\HH)$
and $(\forall x\in\HH)$ $\prox_\varphi x=x-\prox_{\psi/2}(x/2)$.
\end{example}

\begin{example}{\rm\cite[Proposition~11]{Jsts07}} 
\label{ex:1931}
Let $\GG$ be a real Hilbert space, let $\psi\in\Gamma_0(\GG)$, let 
$M\in\BL(\HH,\GG)$, and set $\varphi=\psi\circ M$. Suppose that 
$M\circ M^*=\kappa\Id$, for some $\kappa\in\RPP$. Then 
$\varphi\in\Gamma_0(\HH)$ and
\begin{equation}
\label{e:pfL}
\prox_{\varphi}=
\Id+\frac{1}{\kappa}M^*\circ(\prox_{\kappa\psi}-\Id)\circ M.
\end{equation}
\end{example}

\begin{example}{\rm\cite[Proposition~2.10 and Remark~3.2(ii)]{Chau07}}
\label{ex:chaux07}
Set 
\begin{equation}
\varphi\colon\HH\to\RX\colon x\mapsto\sum_{k\in\KK}\phi_k(\scal{x}{o_k}), 
\end{equation}
where:
\begin{enumerate}
\item
\label{p:quezoni}
$\emp\neq\KK\subset\NN$;
\item
\label{p:quezonii}
$(o_k)_{k\in\KK}$ is an orthonormal basis of $\HH$; 
\item
\label{p:quezoniii}
$(\phi_k)_{k\in\KK}$ are functions in $\Gamma_0(\RR)$;
\item
\label{p:quezoniv}
Either $\KK$ is finite, or there exists a subset $\LL$ of $\KK$ 
such that:
\begin{enumerate}
\item
\label{p:quezoniva}
$\KK\smallsetminus \LL$ is finite;
\item
\label{p:quezonivb}
$(\forall k\in\LL)$ $\phi_k\geq \phi_k(0)=0$.
\end{enumerate}
\end{enumerate}
Then $\varphi\in\Gamma_0(\HH)$ and
\begin{equation}
(\forall x\in\HH)\quad\prox_{\varphi}x=
\sum_{k\in\KK}\big(\prox_{\phi_k}\scal{x}{o_k}\big)o_k.
\end{equation}
\end{example}

\begin{example}
{\rm\cite[Proposition~2.1]{Luis09}}
\label{ex:10}
Let $C$ be a nonempty closed convex subset of $\HH$, let
$\phi\in\Gamma_0(\RR)$ be even, and set $\varphi=\phi\circ d_C$. 
Then $\varphi\in\Gamma_0(\HH)$. Moreover, $\prox_{\varphi}=P_C$ 
if $\phi=\iota_{\{0\}}+\eta$ for some $\eta\in\RR$ and, otherwise,
\begin{equation}
\label{e:26juin2009}
(\forall x\in\HH)\quad\prox_{\varphi}x=
\begin{cases}
x+\displaystyle{\frac{\prox_{\phi^*}d_C(x)}{d_C(x)}}
(P_Cx-x),&\text{if}\;\;d_C(x)>\max\partial\phi(0);\\
P_Cx,&\text{if}\;\;x\notin C\;\text{and}\;\;d_C(x)
\leq\max\partial\phi(0);\\
x,&\text{if}\;\;x\in C.
\end{cases}
\end{equation}
\end{example}

\begin{remark}
\label{r:34}
Taking $C=\{0\}$ and $\phi\neq\iota_{\{0\}}+\eta$ ($\eta\in\RR$)
in Example~\ref{ex:10} yields the proximity operator of 
$\phi\circ\|\cdot\|$, namely
(using Lemma~\ref{l:6}\ref{l:6iii})
\begin{equation}
\label{e:8janvier2009}
(\forall x\in\HH)\quad\prox_{\varphi}x=
\begin{cases}
\displaystyle{\frac{\prox_{\phi}\|x\|}{\|x\|}}x,
&\text{if}\;\;\|x\|>\max\partial\phi(0);\\
0,&\text{if}\;\;\|x\|\leq\max\partial\phi(0).
\end{cases}
\end{equation}
On the other hand, if $\phi$ is differentiable at $0$ in 
Example~\ref{ex:10}, then $\partial\phi(0)=\{0\}$ and 
\eqref{e:26juin2009} yields
\begin{equation}
\label{e:27juin2009}
(\forall x\in\HH)\quad\prox_{\varphi}x=
\begin{cases}
x+\displaystyle{\frac{\prox_{\phi^*}d_C(x)}{d_C(x)}}
(P_Cx-x),&\text{if}\;\;x\notin C;\\
x,&\text{if}\;\;x\in C.
\end{cases}
\end{equation}
\end{remark}

\begin{example}
{\rm\cite[Proposition~2.2]{Luis09}}
\label{ex:10*}
Let $C$ be a nonempty closed convex subset of $\HH$, let
$\phi\in\Gamma_0(\RR)$ be even and nonconstant,
and set $\varphi=\sigma_C+\phi\circ\|\cdot\|$. Then
$\varphi\in\Gamma_0(\HH)$ and 
\begin{equation}
\label{e:2janvier2009}
(\forall x\in\HH)\quad\prox_{\varphi}x=
\begin{cases}
\displaystyle{\frac{\prox_{\phi}d_C(x)}{d_C(x)}}
(x-P_Cx),&\text{if}\;\;d_C(x)>\max\operatorname{Argmin}\phi;\\
x-P_Cx,&\text{if}\;\;x\notin C\;\text{and}\;\;
d_C(x)\leq\max\operatorname{Argmin}\phi;\\
0,&\text{if}\;\;x\in C.
\end{cases}
\end{equation}
\end{example}

\begin{example}
\label{ex:hue2009-01-09}
Let $A\in\BL(\HH)$ be positive and self-adjoint, let $b\in\HH$, 
let $\alpha\in\RR$, and set $\varphi\colon x\mapsto
\scal{Ax}{x}/2+\scal{x}{b}+\alpha$. Then $\varphi\in\Gamma_0(\HH)$ 
and $(\forall x\in\HH)$ $\prox_{\varphi}x=(\Id+A)^{-1}(x-b)$.
\end{example}
\begin{proof}
It is clear that $\varphi$ is a finite-valued continuous convex function.
Now fix $x\in\HH$ and set 
$\psi\colon y\mapsto\|x-y\|^2/2+\scal{Ay}{y}/2+\scal{y}{b}+\alpha$.
Then $\nabla\psi\colon y\mapsto y-x+Ay+b$.
Hence, $(\forall y\in\HH)$ $\nabla\psi(y)=0$ $\Leftrightarrow$
$y=(\Id+A)^{-1}(x-b)$.
\end{proof}

\begin{example}
\label{ex:hanoi2009-01-06}
For every $i\in\{1,\ldots,m\}$, let $(\GG_i,\|\cdot\|)$ be a real 
Hilbert space, let $r_i\in\GG_i$, let $T_i\in\BL(\HH,\GG_i)$, and let
$\alpha_i\in\RPP$. Set $(\forall x\in\HH)$
$\varphi(x)=(1/2)\sum_{i=1}^m\alpha_i\|T_ix-r_i\|^2$.
Then $\varphi\in\Gamma_0(\HH)$ and 
\begin{equation}
\label{e:hanoi2009-01-06}
(\forall x\in\HH)\quad\prox_{\varphi}x=
\bigg(\Id+\sum_{i=1}^m\alpha_iT_i^*T_i\bigg)^{-1}
\bigg(x+\sum_{i=1}^m\alpha_iT_i^*r_i\bigg).
\end{equation}
\end{example}
\begin{proof}
We have $\varphi\colon x\mapsto\sum_{i=1}^m\alpha_i
\scal{T_ix-r_i}{T_ix-r_i}/2=\scal{Ax}{x}/2+\scal{x}{b}+\alpha$,
where $A=\sum_{i=1}^m\alpha_iT_i^*T_i$, 
$b=-\sum_{i=1}^m\alpha_iT_i^*r_i$, 
and $\alpha=\sum_{i=1}^m\alpha_i\|r_i\|^2/2$. 
Hence, \eqref{e:hanoi2009-01-06} 
follows from Example~\ref{ex:hue2009-01-09}.
\end{proof}

As seen in Example~\ref{ex:chaux07}, Example~\ref{ex:10},
Remark~\ref{r:34}, and 
Example~\ref{ex:10*}, some important proximity operators can be 
decomposed in terms of those of functions in $\Gamma_0(\RR)$. 
Here are explicit expressions for the proximity operators of
such functions.

\begin{example} {\rm \cite[Examples~4.2 and~4.4]{Chau07}}
\label{ex:phi}
Let $p\in\left[1,\pinf\right[$, let $\alpha\in\RPP$, let 
$\phi\colon\RR\to\RR\colon\eta\mapsto\alpha|\eta|^{p}$,
let $\xi\in\RR$, and set $\pi=\prox_\phi\xi$. 
Then the following hold.
\begin{enumerate}
\item
\label{ex:phii}
$\pi=\operatorname{sign}(\xi)\max\{|\xi|-\alpha,0\}$, if $p=1$;
\item
\label{ex:phiii}
$\pi=\xi+\Frac{4\alpha}{3\cdot 2^{1/3}}
\Big(|\rho-\xi|^{1/3}-|\rho+\xi|^{1/3}\Big)$,
where $\rho=\sqrt{\xi^2+256\alpha^3/729}$, if $p=4/3$;
\item
\label{ex:phiiii}
$\pi=\xi+9\alpha^2\operatorname{sign}(\xi)
\big(1-\sqrt{1+16|\xi|/(9\alpha^2)}\,\big)/8$,
if $p=3/2$;
\item
\label{ex:phiiv}
$\pi=\xi/(1+2\alpha)$, if $p=2$;
\item
\label{ex:phiv}
$\pi=\operatorname{sign}(\xi)\big(\sqrt{1+12\alpha|\xi|}-1\big)
/(6\alpha)$, if $p=3$;
\item
\label{ex:phivi}
$\pi=\bigg|\Frac{\rho+\xi}{8\alpha}\bigg|^{1/3}-
\bigg|\Frac{\rho-\xi}{8\alpha}\bigg|^{1/3}$,
where $\rho=\sqrt{\xi^2+1/(27\alpha)}$, if $p=4$.
\end{enumerate}
\end{example}

\begin{example}{\rm\cite[Example~2.18]{Smms05}}
\label{ex:36}
Let $\alpha\in\RPP$ and set 
\begin{equation}
\phi\colon\xi\mapsto
\begin{cases}
-\alpha\ln(\xi),&\text{if}\;\;\xi>0;\\
\pinf,&\text{if}\;\;\xi\leq 0.
\end{cases}
\end{equation}
Then $(\forall\xi\in\RR)$ 
$\prox_{\phi}\xi=(\xi+\sqrt{\xi^2+4\alpha})/2$.
\end{example}

\begin{example}{\rm\cite[Example~3.5]{Siop07}}
\label{ex:7}
Let $\omega\in\RPP$ and set 
\begin{equation}
\label{e:td}
\phi\colon\RR\to\RX\colon\xi\mapsto 
\begin{cases}
\ln(\omega)-\ln(\omega-|\xi|),
&\text{if}\;\;|\xi|<\omega;\\
\pinf, &\text{otherwise.}
\end{cases}
\end{equation}
Then
\begin{equation}
(\forall\xi\in\RR)\quad\prox_\phi\xi=
\begin{cases}
\sign(\xi)\,\displaystyle\frac{|\xi|+\omega-
\sqrt{\big||\xi|-\omega\big|^2+4}}{2}, 
& \text{if}\;\;|\xi|>1/\omega;\\[2mm]
0 & \text{otherwise}.
\end{cases}
\end{equation}
\end{example}

\begin{example}{\rm\cite[Example~4.5]{Chau07}}
\label{ex:8}
Let $\omega\in\RPP$, $\tau\in\RPP$, and set 
\begin{equation}
\label{e:fhd}
\phi\colon\RR\to\RX\colon\xi\mapsto \begin{cases}
\tau\xi^2, & \text{if}\;\;|\xi|\leq\omega/\sqrt{2\tau};\\[3mm]
\omega\sqrt{2\tau}|\xi|-
{\omega^2}/{2}, & \text{otherwise}.
\end{cases}
\end{equation}
Then
\begin{equation}
(\forall\xi\in\RR)\quad\prox_\phi\xi=
\begin{cases}
\Frac{\xi}{2\tau+1}, 
&\text{if}\;\;|\xi|\leq\omega(2\tau+1)/\sqrt{2\tau};\\[3mm]
\xi-\omega\sqrt{2\tau}\sign(\xi), 
&\text{if}\;\;|\xi|>\omega(2\tau+1)/\sqrt{2\tau}.
\end{cases}
\end{equation}
\end{example}

Further examples can be constructed via the following rules.

\begin{lemma}{\rm\cite[Proposition~3.6]{Siop07}}
\label{l:santiago06-1}
Let $\phi=\psi+\sigma_\Omega$, where $\psi\in\Gamma_0(\RR)$ and 
$\Omega\subset\RR$ is a nonempty closed interval. Suppose that $\psi$ is 
differentiable at $0$ with $\psi'(0)=0$. Then 
$\prox_\phi=\prox_\psi\circ\,\soft{\Omega}$, where
\begin{equation}
\label{e:soft4}
\soft{\Omega}\colon\RR\to\RR\colon\xi\mapsto
\begin{cases}
\xi-\underline{\omega},&\text{if}\;\;\xi<\underline{\omega};\\
0,&\text{if}\;\;\xi\in\Omega;\\
\xi-\overline{\omega},&\text{if}\;\;\xi>\overline{\omega},
\end{cases}
\quad
\qquad\text{with}\quad
\begin{cases}
\underline{\omega}=\inf\Omega,\\
\overline{\omega}=\sup\Omega.
\end{cases}
\end{equation}
\end{lemma}

\begin{lemma}{\rm\cite[Proposition~12(ii)]{Jsts07}}
\label{l:3}
Let $\phi=\iota_C+\psi$, where $\psi\in\Gamma_0(\RR)$ and where $C$ is 
a closed interval in $\RR$ such that $C\cap\dom\psi\neq\emp$. Then 
$\prox_{\iota_C+\psi}=P_C\circ\prox_\psi$.
\end{lemma}

\section{Dualization and algorithm}
\label{sec:2}

\subsection{Fenchel-Moreau-Rockafellar duality}

Our analysis will revolve around the following version of the
Fenchel-Moreau-Rockafellar duality formula (see \cite{Fenc53}, 
\cite{More66}, and \cite{Rock67} for historical work).
It will also exploit various aspects of the Baillon-Haddad
theorem \cite{Joca10}.

\begin{lemma} {\rm\cite[Corollary~2.8.5]{Zali02}}
\label{l:5}
Let $\varphi\in\Gamma_0(\HH)$, let $\psi\in\Gamma_0(\GG)$, and let 
$M\in\BL(\HH,\GG)$ be such that $0\in\sri(M(\dom\varphi)-\dom\psi)$. Then 
\begin{equation}
\label{e:F-R}
\inf_{x\in\HH}\:\varphi(x)+\psi(Mx)=-
\min_{v\in\GG}\:\varphi^*(-M^*v)+\psi^*(v).
\end{equation}
\end{lemma}

The problem of minimizing $\varphi+\psi\circ M$ on $\HH$ in \eqref{e:F-R} 
is referred to as the primal problem, and that of minimizing
$\varphi^*\circ(-M^*)+\psi^*$ on $\GG$ as the dual problem. 
Lemma~\ref{l:5} gives conditions under which a dual solution exists
and the value of the dual problem coincides with the opposite of the 
value of the primal problem. We can now introduce the dual of 
Problem~\ref{prob:1}.

\begin{problem}[dual problem]
\label{prob:2}
Under the same assumptions as in Problem~\ref{prob:1}, 
\begin{equation}
\label{e:prob2}
\underset{v\in\GG}{\mathrm{minimize}}\;\;
\widetilde{f^*}(z-L^*v)+g^*(v)+\scal{v}{r}.
\end{equation}
\end{problem}

\begin{proposition}
\label{p:10novembre2008}
Problem~\ref{prob:2} is the dual of Problem~\ref{prob:1} and it 
admits at least one solution. Moreover, every solution $v$ to 
Problem~\ref{prob:2} is characterized by the inclusion 
\begin{equation}
\label{e:25juin2009}
L\big(\prox_{f}(z-L^*v)\big)-r\in\partial g^*(v).
\end{equation}
\end{proposition}
\begin{proof}
Let us set $w=z$, $\varphi=f+\|\cdot-w\|^2/2$, $M=L$, and
$\psi=g(\cdot-r)$. Then $(\forall x\in\HH)$ $\varphi(x)+\psi(Mx)=
f(x)+g(Lx-r)+\|x-z\|^2/2$. Hence, it results from \eqref{e:F-R}
and Lemma~\ref{l:2} that the dual of Problem~\ref{prob:1} is 
to minimize the function 
\begin{align}
\varphi^*\circ(-M^*)+\psi^*\colon v
&\mapsto\widetilde{f^*}(-M^*v+w)-\frac12\|w\|^2+\psi^*(v)\nonumber\\
&=\widetilde{f^*}(z-L^*v)-\frac12\|z\|^2+g^*(v)+\scal{v}{r}
\end{align}
or, equivalently, the function
$v\mapsto\widetilde{f^*}(z-L^*v)+g^*(v)+\scal{v}{r}$. In view of 
\eqref{e:1938}, the first two claims therefore follow from 
Lemma~\ref{l:5}. To establish the last claim, note that \eqref{e:moreau1} 
asserts that $\dom\widetilde{f^*}\circ(z-L^*\cdot)=\GG$. Hence, using 
\eqref{e:fermat}, Lemma~\ref{l:4}, \eqref{e:moreau23}, and 
Lemma~\ref{l:6}\ref{l:6iii}, we get
\begin{eqnarray}
\label{e:2009-06-25b}
v~\text{solves~\eqref{e:prob2}}
&\Leftrightarrow&
0\in\partial\Big(\widetilde{f^*}\circ(z-L^*\cdot)+g^*
+\scal{\cdot}{r}\Big)(v)\nonumber\\
&\Leftrightarrow&0\in -L\big(\nabla\widetilde{f^*}(z-L^*v)\big)
+\partial g^*(v)+r\nonumber\\
&\Leftrightarrow&0\in -L\big(\prox_{f}(z-L^*v)\big)
+\partial g^*(v)+r,
\end{eqnarray}
which yields \eqref{e:25juin2009}.
\end{proof}

A key property underlying our setting is that the primal solution 
can actually be recovered from any dual solution (this is property (c)
in the Introduction).

\begin{proposition}
\label{p:2}
Let $v$ be a solution to Problem~\ref{prob:2} and set
\begin{equation}
\label{e:athenes2008}
x=\prox_f(z-L^*v).
\end{equation}
Then $x$ is the solution to Problem~\ref{prob:1}.
\end{proposition}
\begin{proof}
We derive from \eqref{e:athenes2008} and \eqref{e:prox1} that 
$z-L^*v-x\in\partial f(x)$. Therefore
\begin{equation}
\label{e:2008-11-10a}
-L^*v\in\partial f(x)+x-z.
\end{equation}
On the other hand, it follows from \eqref{e:25juin2009}, 
\eqref{e:athenes2008}, and \eqref{e:terminal2E} that
\begin{eqnarray}
\label{e:2008-11-10b}
v~\text{solves~\eqref{e:prob2}}
&\Leftrightarrow&Lx-r\in\partial g^*(v)\nonumber\\
&\Leftrightarrow&v\in\partial g(Lx-r)
\nonumber\\
&\Rightarrow&L^*v\in L^*\big(\partial g(Lx-r)\big).
\end{eqnarray}
Upon adding \eqref{e:2008-11-10a} and \eqref{e:2008-11-10b},
invoking Lemma~\ref{l:4}, and then \eqref{e:fermat} we obtain
\begin{align}
\label{e:2008-11-10c}
v~\text{solves~\eqref{e:prob2}}\quad\Rightarrow\quad
0
&=L^*v-L^*v\nonumber\\
&\in\partial f(x)+L^*\big(\partial g(Lx-r)\big)+x-z\nonumber\\
&=\partial f(x)+L^*\big(\partial g(Lx-r)\big)+
\nabla\Big(\frac12\|\cdot-z\|^2\Big)(x)\nonumber\\
&=\partial\Big(f+g(L\cdot-r)+\frac12\|\cdot-z\|^2\Big)(x)\nonumber\\
\quad\Leftrightarrow\quad
x&~\text{solves~\eqref{e:prob1}},
\end{align}
which completes the proof.
\end{proof}

\subsection{Algorithm}

As seen in \eqref{e:rio-mai2009}, the unique solution to
Problem~\ref{prob:1} is $\prox_{h}z$, where $h$ is defined in
\eqref{e:24juin2009}. Since $\prox_{h}z$ cannot be computed 
directly, it will be constructed iteratively by the following 
algorithm, which produces a primal sequence $(x_n)_{n\in\NN}$ as well
as a dual sequence $(v_n)_{n\in\NN}$. 

\begin{algorithm}
\label{algo:1}
Let $(a_{n})_{n\in\NN}$ be a sequence in $\GG$ such that
$\sum_{n\in\NN}\|a_n\|<\pinf$ and let $(b_{n})_{n\in\NN}$ be a 
sequence in $\HH$ such that $\sum_{n\in\NN}\|b_n\|<\pinf$. 
Sequences $(x_n)_{n\in\NN}$ and $(v_n)_{n\in\NN}$ are generated by 
the following routine.
\begin{equation}
\label{e:main1}
\begin{array}{l}
\operatorname{Initialization}\\
\left\lfloor
\begin{array}{l}
\varepsilon\in\left]0,\min\{1,\|L\|^{-2}\}\right[\\[1mm]
v_0\in\GG\\[1mm]
\end{array}
\right.\\[5mm]
\operatorname{For}\;n=0,1,\ldots\\
\left\lfloor
\begin{array}{l}
x_n=\prox_f(z-L^*v_n)+b_n\\[1mm]
\gamma_n\in\left[\varepsilon,2\|L\|^{-2}-\varepsilon\right]\\[1mm]
\lambda_n\in\left[\varepsilon,1\right]\\
v_{n+1}=v_n+\lambda_n\big(\prox_{\gamma_n g^*}(v_n
+\gamma_n(Lx_n-r))+a_n-v_n\big).
\end{array}
\right.\\[2mm]
\end{array}
\end{equation}
\end{algorithm}

It is noteworthy that each iteration of Algorithm~\ref{algo:1} 
achieves full splitting with respect to the operators $L$, $\prox_f$, 
and $\prox_{g^*}$, which are used at separate steps. In addition, 
\eqref{e:main1} incorporates tolerances $a_n$ and $b_n$ in the 
computation of the proximity operators at iteration $n$.

\subsection{Convergence}

Our main convergence result will be a consequence of 
Proposition~\ref{p:2} and the following
results on the convergence of the forward-backward splitting method. 

\begin{theorem}{\rm\cite[Theorem~3.4]{Smms05}}
\label{t:0}
Let $f_1$ and $f_2$ be functions in $\Gamma_0(\GG)$ such that 
the set $G$ of minimizers of $f_1+f_2$ is nonempty and such 
that $f_2$ is differentiable on $\GG$ with a $1/\beta$-Lipschitz 
continuous gradient for some $\beta\in\RPP$. Let $(\gamma_n)_{n\in\NN}$ 
be a sequence in $\left]0,2\beta\right[$ such that 
$\inf_{n\in\NN}\gamma_n>0$ and $\sup_{n\in\NN}\gamma_n<2\beta$, let 
$(\lambda_n)_{n\in\NN}$ be a sequence in $\left]0,1\right]$ such that 
$\inf_{n\in\NN}\lambda_n>0$, and let $(a_{1,n})_{n\in\NN}$ and
$(a_{2,n})_{n\in\NN}$ be sequences in $\GG$ such that 
$\sum_{n\in\NN}\|a_{1,n}\|<\pinf$ and $\sum_{n\in\NN}\|a_{2,n}\|<\pinf$.
Fix $v_0\in\GG$ and, for every $n\in\NN$, set
\begin{equation}
\label{e:main2005}
v_{n+1}=v_n+\lambda_n\Big(\prox_{\gamma_n f_1}
\big(v_n-\gamma_n(\nabla f_2(v_n)+a_{2,n})\big)+a_{1,n}-v_n\Big).
\end{equation}
Then $(v_n)_{n\in\NN}$ converges weakly to a point $v\in G$
and $\sum_{n\in\NN}\big\|\nabla f_2(v_n)-\nabla f_2(v)\|^2<\pinf$.
\end{theorem}

The following theorem describes the asymptotic behavior of 
Algorithm~\ref{algo:1}.

\begin{theorem}
\label{t:1}
Let $(x_n)_{n\in\NN}$ and $(v_n)_{n\in\NN}$ be sequences generated 
by Algorithm~\ref{algo:1}, and let $x$ be the solution to 
Problem~\ref{prob:1}. Then the following hold.
\begin{enumerate}
\item
\label{t:1i}
$(v_n)_{n\in\NN}$ converges weakly to a solution $v$ to 
Problem~\ref{prob:2} and $x=\prox_f(z-L^*v)$.
\item
\label{t:1x}
$(x_n)_{n\in\NN}$ converges strongly to $x$.
\end{enumerate}
\end{theorem}
\begin{proof}
Let us define two functions $f_1$ and $f_2$ on $\GG$ by
$f_1\colon v\mapsto g^*(v)+\scal{v}{r}$ and 
$f_2\colon v\mapsto\widetilde{f^*}(z-L^*v)$. Then \eqref{e:prob2} 
amounts to minimizing $f_1+f_2$ on $\GG$. Let us first check that all 
the assumptions specified in Theorem~\ref{t:0} are satisfied. First, 
$f_1$ and $f_2$ are in $\Gamma_0(\GG)$ and, by 
Proposition~\ref{p:10novembre2008}, 
$\operatorname{Argmin}f_1+f_2\neq\emp$. Moreover, it follows from 
Lemma~\ref{l:6}\ref{l:6iii} that $f_2$ is differentiable on $\GG$ 
with gradient
\begin{equation}
\label{e:f2'}
\nabla f_2\colon v\mapsto -L\big(\prox_{f}(z-L^*v)\big).
\end{equation}
Hence, we derive from Lemma~\ref{l:6}\ref{l:6ii} that 
\begin{align}
(\forall v\in\GG)(\forall w\in\GG)\quad
\|\nabla f_2(v)-\nabla f_2(w)\|
&\leq\|L\|\,\|\prox_{f}(z-L^*v)-\prox_{f}(z-L^*w)\|\nonumber\\
&\leq\|L\|\,\|L^*v-L^*w\|\nonumber\\
&\leq\|L\|^2\,\|v-w\|.
\end{align}
The reciprocal of the Lipschitz constant of $\nabla f_2$ is therefore
$\beta=\|L\|^{-2}$. Now set 
\begin{equation}
(\forall n\in\NN)\quad a_{1,n}=a_n\quad\text{and}\quad a_{2,n}=-Lb_n. 
\end{equation}
Then $\sum_{n\in\NN}\|a_{1,n}\|=\sum_{n\in\NN}\|a_n\|<\pinf$ and 
$\sum_{n\in\NN}\|a_{2,n}\|\leq\|L\|\sum_{n\in\NN}\|b_{n}\|<\pinf$.
Moreover, for every $n\in\NN$, \eqref{e:main1} yields
\begin{equation}
\label{e:VN}
x_n=\prox_{f}(z-L^*v_n)+b_n
\end{equation}
and, together with \cite[Lemma~2.6(i)]{Smms05},
\begin{align}
\label{e:main2008}
v_{n+1}
&=v_n+\lambda_n\Big(\prox_{\gamma_n g^*}
\big(v_n+\gamma_n(Lx_n-r)\big)+a_{n}-v_n\Big)\nonumber\\
&=v_n+\lambda_n\Big(\prox_{\gamma_n g^*+\scal{\cdot}{\gamma_nr}}
\big(v_n+\gamma_nLx_n\big)+a_{n}-v_n\Big)\nonumber\\
&=v_n+\lambda_n\Big(\prox_{\gamma_n(g^*+\scal{\cdot}{r})}
\big(v_n+\gamma_nL(\prox_{f}(z-L^*v_n)+b_n)\big)
+a_{n}-v_n\Big)\nonumber\\
&=v_n+\lambda_n\Big(\prox_{\gamma_n f_1}
\big(v_n-\gamma_n(\nabla f_2(v_n)+a_{2,n})\big)
+a_{1,n}-v_n\Big).
\end{align}
This provides precisely the update rule \eqref{e:main2005}, 
which allows us to apply Theorem~\ref{t:0}.

\ref{t:1i}: In view of the above, we derive from 
Theorem~\ref{t:0} that $(v_n)_{n\in\NN}$ converges weakly
to a solution $v$ to \eqref{e:prob2}. The second assertion follows
from Proposition~\ref{p:2}.

\ref{t:1x}: Let us set
\begin{equation}
\label{e:b1}
(\forall n\in\NN)\quad y_n=x_n-b_n=\prox_{f}(z-L^*v_n).
\end{equation}
As seen in \ref{t:1i}, $v_n\weakly v$, where $v$ is a solution
to \eqref{e:prob2}, and $x=\prox_f(z-L^*v)$. 
Now set $\rho=\sup_{n\in\NN}\|v_n-v\|$.
Then $\rho<\pinf$ and, using Lemma~\ref{l:6}\ref{l:6ii-} 
and \eqref{e:f2'}, we obtain
\begin{align}
\label{e:b2}
\|y_n-x\|^2
&=\|\prox_{f}(z-L^*v_n)-\prox_{f}(z-L^*v)\|^2\nonumber\\
&\leq\scal{L^*v-L^*v_n}
{\prox_{f}(z-L^*v_n)-\prox_{f}(z-L^*v)}\nonumber\\
&=\scal{v_n-v}
{-L\big(\prox_{f}(z-L^*v_n)\big)+L\big(\prox_{f}(z-L^*v)\big)}\nonumber\\
&=\scal{v_n-v}{\nabla f_2(v_n)-\nabla f_2(v)}\nonumber\\
&\leq\rho\|\nabla f_2(v_n)-\nabla f_2(v)\|.
\end{align}
However, as seen in Theorem~\ref{t:0}, 
$\|\nabla f_2(v_n)-\nabla f_2(v)\|\to 0$. Hence, we derive from
\eqref{e:b2} that $y_n\to x$. In turn, since $b_n\to 0$, 
\eqref{e:b1} yields $x_n\to x$.
\end{proof}

\begin{remark}[Dykstra-like algorithm]
Suppose that, in Problem~\ref{prob:1}, $\GG=\HH$, $L=\Id$, and $r=0$.
Then it follows from Theorem~\ref{t:1}\ref{t:1x} that the sequence
$(x_n)_{n\in\NN}$ produced by Algorithm~\ref{algo:1} converges strongly
to $x=\prox_{f+g}z$. Now let us consider the special case when
Algorithm~\ref{algo:1} is implemented with $v_0=0$, 
$\gamma_n\equiv1$, $\lambda_n\equiv 1$, and no errors, i.e.,
$a_n\equiv 0$ and $b_n\equiv 0$. Then it follows from 
Lemma~\ref{l:6}\ref{l:6iii} that \eqref{e:main1} simplifies to
\begin{equation}
\label{e:main1+}
\begin{array}{l}
\operatorname{Initialization}\\
\left\lfloor
\begin{array}{l}
v_0=0\\[0mm]
\end{array}
\right.\\[1mm]
\operatorname{For}\;n=0,1,\ldots\\
\left\lfloor
\begin{array}{l}
x_n=\prox_f(z-v_n)\\[1mm]
v_{n+1}=x_n+v_n-\prox_{g}(x_n+v_n).
\end{array}
\right.\\[2mm]
\end{array}
\end{equation}
Using \cite[Eq.~(2.10)]{Pjo208} it can then easily be shown by
induction that the resulting sequence $(x_n)_{n\in\NN}$ coincides 
with that produced by the Dykstra-like algorithm \eqref{e:6ans} 
(with $h_1=g$ and $h_2=f$) and that the sequence $(v_n)_{n\in\NN}$ 
coincides with the sequence $(p_n)_{n\in\NN}$ of \eqref{e:6ans}. 
The fact that $x_n\to\prox_{f+g}z$ was established in 
\cite[Theorem~3.3(i)]{Pjo208} using different tools. 
Thus, Algorithm~\ref{algo:1} can be regarded as a generalization 
of the Dykstra-like algorithm \eqref{e:6ans}.
\end{remark}

\begin{remark}
Theorem~\ref{t:1} remains valid if we introduce explicitly
errors in the implementation of the operators $L$ and $L^*$ in
Algorithm~\ref{algo:1}. More precisely, we can replace the steps
defining $x_n$ and $v_n$ in \eqref{e:main1} by
\begin{equation}
\label{e:main9}
\begin{array}{l}
\left\lfloor
\begin{array}{l}
x_n=\prox_f(z-L^*v_n-d_{2,n})+d_{1,n}\\[1mm]
v_{n+1}=v_n+\lambda_n\big(\prox_{\gamma_n g^*}(v_n
+\gamma_n(Lx_n+c_{2,n}-r))+c_{1,n}-v_n\big),
\end{array}
\right.\\[2mm]
\end{array}
\end{equation}
where $(d_{1,n})_{n\in\NN}$ and $(d_{2,n})_{n\in\NN}$ are 
sequences in $\HH$ such that $\sum_{n\in\NN}\|d_{1,n}\|<\pinf$ 
and $\sum_{n\in\NN}\|d_{2,n}\|<\pinf$, and where 
$(c_{1,n})_{n\in\NN}$ and $(c_{2,n})_{n\in\NN}$ are 
sequences in $\GG$ such that
$\sum_{n\in\NN}\|c_{1,n}\|<\pinf$ and 
$\sum_{n\in\NN}\|c_{2,n}\|<\pinf$.
Indeed set, for every $n\in\NN$,
\begin{equation}
\begin{cases}
a_n=c_{1,n}+\prox_{\gamma_n g^*}(v_n
+\gamma_n(Lx_n+c_{2,n}-r))-\prox_{\gamma_n g^*}(v_n
+\gamma_n(Lx_n-r))\\
b_n=d_{1,n}+\prox_f(z-L^*v_n-d_{2,n})-\prox_f(z-L^*v_n).
\end{cases}
\end{equation}
Then \eqref{e:main9} reverts to 
\begin{equation}
\label{e:main9j}
\begin{array}{l}
\left\lfloor
\begin{array}{l}
x_n=\prox_f(z-L^*v_n)+b_{n}\\[1mm]
v_{n+1}=v_n+\lambda_n\big(\prox_{\gamma_n g^*}(v_n
+\gamma_n(Lx_n-r))+a_{n}-v_n\big),
\end{array}
\right.\\[2mm]
\end{array}
\end{equation}
as in \eqref{e:main1}. Moreover, by Lemma~\ref{l:6}\ref{l:6ii},
\begin{align}
(\forall n\in\NN)\quad
\|a_n\|
&\leq\|c_{1,n}\|+\|\prox_{\gamma_n g^*}(v_n
+\gamma_n(Lx_n+c_{2,n}-r))-\prox_{\gamma_n g^*}(v_n
+\gamma_n(Lx_n-r))\|\nonumber\\
&\leq\|c_{1,n}\|+\gamma_n\|c_{2,n}\|\nonumber\\
&\leq\|c_{1,n}\|+2\|L\|^{-2}\|c_{2,n}\|.
\end{align}
Thus, $\sum_{n\in\NN}\|a_{n}\|<\pinf$. Likewise, we have
$\sum_{n\in\NN}\|b_{n}\|<\pinf$.
\end{remark}

\section{Application to specific signal recovery problems}
\label{sec:3}

In this section, we present a few applications of the duality 
framework presented in Section~\ref{sec:2}, which correspond to 
specific choices of $\HH$, $\GG$, $L$, $f$, $g$, $r$, and $z$ in 
Problem~\ref{prob:1}.

\subsection{Best feasible approximation}
\label{sec:31}

A standard feasibility problem in signal recovery is to find 
a signal in the intersection of two closed convex sets modeling
constraints on the ideal solution \cite{Aiep96,Star87,Trus84,Youl82}.
A more structured variant of this problem, is the so-called split
feasibility problem \cite{Byrn05,Cens94,Cens97}, which requires to 
find a signal in a closed convex set $C\subset\HH$ and such that 
some affine transformation of it lies in a closed convex set 
$D\subset\GG$. Such problems typically admit infinitely many 
solutions and one often seeks to find the solution that lies 
closest to a nominal signal $z\in\HH$ \cite{Imag93,Pott93}. 
This leads to the formulation \eqref{e:porquerolles2009-06-12}, 
which consists in finding the best approximation to a reference signal 
$z\in\HH$ from the feasibility set $C\cap L^{-1}(r+D)$.

\begin{problem}
\label{prob:91}
Let $z\in\HH$, let $r\in\GG$, let $C\subset\HH$ and $D\subset\GG$ 
be closed convex sets, and let $L$ be a nonzero operator in 
$\BL(\HH,\GG)$ such that 
\begin{equation}
\label{e:1942}
r\in\sri\big(L(C)-D\big).
\end{equation}
The problem is to 
\begin{equation}
\label{e:prob4}
\underset{\substack{x\in C\\Lx-r\in D}}{\mathrm{minimize}}\;\;
\frac12\|x-z\|^2,
\end{equation}
and its dual is to
\begin{equation}
\label{e:prob4'}
\underset{v\in\GG}{\mathrm{minimize}}\;\;
\frac12\|z-L^*v\|^2-\frac12d_C^2(z-L^*v)+\sigma_D(v)+\scal{v}{r}.
\end{equation}
\end{problem}

\begin{proposition}
\label{p:91}
Let $(b_{n})_{n\in\NN}$ be a sequence in $\HH$ such that
$\sum_{n\in\NN}\|b_n\|<\pinf$, let $(c_{n})_{n\in\NN}$ be a 
sequence in $\GG$ such that $\sum_{n\in\NN}\|c_n\|<\pinf$,
and let $(x_n)_{n\in\NN}$ and $(v_n)_{n\in\NN}$ be sequences 
generated by the following routine.
\begin{equation}
\label{e:main911}
\begin{array}{l}
\operatorname{Initialization}\\
\left\lfloor
\begin{array}{l}
\varepsilon\in\left]0,\min\{1,\|L\|^{-2}\}\right[\\[1mm]
v_0\in\GG\\[1mm]
\end{array}
\right.\\[5mm]
\operatorname{For}\;n=0,1,\ldots\\
\left\lfloor
\begin{array}{l}
x_n=P_C(z-L^*v_n)+b_n\\[1mm]
\gamma_n\in\left[\varepsilon,2\|L\|^{-2}-\varepsilon\right]\\[1mm]
\lambda_n\in\left[\varepsilon,1\right]\\
v_{n+1}=v_n+\lambda_n\gamma_n
\big(Lx_n-r-P_D(\gamma_n^{-1}v_n+Lx_n-r)+c_n\big).
\end{array}
\right.\\[2mm]
\end{array}
\end{equation}
Then the following hold, where $x$ designates the primal
solution to Problem~\ref{prob:91}.
\begin{enumerate}
\item
\label{p:91i}
$(v_n)_{n\in\NN}$ converges weakly to a solution $v$ to 
\eqref{e:prob4'} and $x=P_C(z-L^*v)$.
\item
\label{p:91x}
$(x_n)_{n\in\NN}$ converges strongly to $x$.
\end{enumerate}
\end{proposition}
\begin{proof}
Set $f=\iota_C$ and $g=\iota_D$. Then \eqref{e:prob1} reduces to
\eqref{e:prob4} and \eqref{e:1938} reduces to \eqref{e:1942}.
In addition, we derive from Lemma~\ref{l:6}\ref{l:6i} that 
$\widetilde{f^*}=\|\cdot\|^2/2-\widetilde{\iota_C}=
(\|\cdot\|^2-d_C^2)/2$. Hence, in view of \eqref{e:prob2},
\eqref{e:prob4'} in indeed the dual of \eqref{e:prob4}. Furthermore,
items \ref{ex:1938i} and \ref{ex:1938ii} in
Example~\ref{ex:1938} yield $\prox_f=P_C$ and 
\begin{equation}
(\forall n\in\NN)\quad
\prox_{\gamma_n g^*}=
\prox_{\gamma_n\sigma_D}=
\prox_{\sigma_{\gamma_nD}}=
\Id-P_{\gamma_nD}=\Id-\,\gamma_nP_D(\cdot/\gamma_n).
\end{equation}
Finally, set $(\forall n\in\NN)$ $a_n=\gamma_nc_n$. Then 
$\sum_{n\in\NN}\|a_n\|\leq 2\|L\|^{-2}\sum_{n\in\NN}\|c_n\|<\pinf$
and, altogether, \eqref{e:main1} reduces to \eqref{e:main911}. 
Hence, the results follow from Theorem~\ref{t:1}.
\end{proof}

Our investigation was motivated in the Introduction by the duality
framework of \cite{Pott93}. In the next example we recover and sharpen
Proposition~\ref{p:lee93}.

\begin{example}
\label{ex:91}
Consider the special case of Problem~\ref{prob:91} in which 
$z=0$, $\GG=\RR^N$, $D=\{0\}$, $r=(\rho_i)_{1\leq i\leq N}$, and 
$L\colon x\mapsto(\scal{x}{s_i})_{1\leq i\leq N}$, where 
$(s_i)_{1\leq i\leq N}\in\HH^N$ satisfies 
$\sum_{i=1}^N\|s_i\|^2\leq1$. Then, by
\eqref{e:2009-06-30}, \eqref{e:1942} reduces to 
$r\in\rint L(C)$ and \eqref{e:prob4} to \eqref{e:1965}.
Since $\|L\|\leq 1$, specializing \eqref{e:main911} to the case when 
$c_n\equiv 0$ and $\lambda_n\equiv 1$, and introducing the sequence 
$(w_n)_{n\in\NN}=(-v_n)_{n\in\NN}$ for convenience yields
the following routine.
\begin{equation}
\label{e:main912}
\begin{array}{l}
\operatorname{Initialization}\\
\left\lfloor
\begin{array}{l}
\varepsilon\in\left]0,1\right[\\[1mm]
w_0\in\RR^N\\[1mm]
\end{array}
\right.\\[5mm]
\operatorname{For}\;n=0,1,\ldots\\
\left\lfloor
\begin{array}{l}
x_n=P_C(L^*w_n)+b_n\\[1mm]
\gamma_n\in\left[\varepsilon,2\|L\|^{-2}-\varepsilon\right]\\[1mm]
w_{n+1}=w_n+\gamma_n\big(r-Lx_n\big).
\end{array}
\right.\\[2mm]
\end{array}
\end{equation}
Thus, if $\sum_{n\in\NN}\|b_n\|<\pinf$, we deduce from
Proposition~\ref{p:91}\ref{p:91i} and 
Proposition~\ref{p:10novembre2008} the weak convergence of 
$(w_n)_{n\in\NN}$ to a point $w$ such that $v=-w$ satisfies
\eqref{e:25juin2009}, i.e.,
$L(P_C(-L^*v))-r\in\partial\iota^*_{\{0\}}(v)=\{0\}$ or,
equivalently, $L(P_C(L^*w))=r$, and such that $P_C(-L^*v)=P_C(L^*w)$
is the solution to \eqref{e:1965}. In addition, we derive from 
Proposition~\ref{p:91}\ref{p:91x}, the strong convergence of
$(x_n)_{n\in\NN}$ to the solution to \eqref{e:1965}. These results
sharpen the conclusion of Proposition~\ref{p:lee93} (note that 
\eqref{e:main91} corresponds to setting $b_n\equiv 0$ and 
$\gamma_n\equiv\gamma\in\left]0,2\right[$ in \eqref{e:main912}).
\end{example}

\begin{example}
\label{ex:92}
We consider the standard linear inverse problem of recovering an ideal 
signal $\overline{x}\in\HH$ from an observation
\begin{equation}
\label{e:model2}
r=L\overline{x}+s
\end{equation}
in $\GG$, where $L\in\BL(\HH,\GG)$ and where $s\in\GG$ models noise.
Given an estimate $x$ of $\overline{x}$, the residual 
$r-Lx$ should ideally behave like the noise process. Thus, any known 
probabilistic attribute of the noise process can give rise to a 
constraint. This observation was used in \cite{Sign91,Trus84} to 
construct various constraints of the type $Lx-r\in D$, where $D$ is 
closed and convex. In this context, \eqref{e:prob4}
amounts to finding the signal which is closest to some 
nominal signal $z$ and which satisfies a noise-based 
constraint and some convex constraint on 
$\overline{x}$ represented by $C$. Such problems were considered for
instance in \cite{Imag93}, where they were solved by methods that
require the projection onto the set $\menge{x\in\HH}{Lx-r\in D}$, which
is typically hard to compute, even in the simple case when $D$ 
is a closed Euclidean ball \cite{Trus84}. By contrast, the iterative 
method \eqref{e:main911} requires only the projection onto $D$ 
to enforce such constraints. 
\end{example}

\subsection{Soft best feasible approximation}

It follows from \eqref{e:1942} that the underlying feasibility set
$C\cap L^{-1}(r+D)$ in Problem~\ref{prob:91} is nonempty. 
In many situations, feasibility may 
not guaranteed due to, for instance, imprecise prior information or 
unmodeled dynamics in the data formation process \cite{Sign94,Youl86}.
In such instances, one can relax the hard constraints
$x\in C$ and $Lx-r\in D$ in \eqref{e:prob4} by merely forcing that
$x$ be close to $C$ and $Lx-r$ be close to $D$. Let us formulate this
problem within the framework of Problem~\ref{prob:1}.

\begin{problem}
\label{prob:92}
Let $z\in\HH$, let $r\in\GG$, let $C\subset\HH$ and $D\subset\GG$ be 
nonempty closed convex sets, let $L\in\BL(\HH,\GG)$ be a nonzero 
operator, and let $\phi$ and $\psi$ be even functions in 
$\Gamma_0(\RR)\smallsetminus\{\iota_{\{0\}}\}$ such that
\begin{equation}
\label{e:1968}
r\in\sri\big(L\big(\menge{x\in\HH}{d_C(x)\in\dom\phi}\big)-
\menge{y\in\GG}{d_D(y)\in\dom\psi}\big).
\end{equation}
The problem is to 
\begin{equation}
\label{e:prob44}
\underset{x\in\HH}{\mathrm{minimize}}\;\;\phi\big(d_C(x)\big)
+\psi\big(d_D(Lx-r)\big)+\frac12\|x-z\|^2,
\end{equation}
and its dual is to
\begin{equation}
\label{e:prob44'}
\underset{v\in\GG}{\mathrm{minimize}}\;\;
\frac12\|z-L^*v\|^2-{(\phi\circ d_C)^\sim}(z-L^*v)+
\sigma_D(v)+\psi^*(\|v\|)+\scal{v}{r}.
\end{equation}
\end{problem}

Since $\phi$ and $\psi$ are even functions in 
$\Gamma_0(\RR)\smallsetminus\{\iota_{\{0\}}\}$, 
we can use Example~\ref{ex:10} to get an explicitly expression of 
the proximity operators involved and solve the minimization 
problems \eqref{e:prob44} and \eqref{e:prob44'} as follows. 

\begin{proposition}
\label{p:92}
Let $(b_{n})_{n\in\NN}$ be a sequence in $\HH$ such that
$\sum_{n\in\NN}\|b_n\|<\pinf$, let $(c_{n})_{n\in\NN}$ be a 
sequence in $\GG$ such that $\sum_{n\in\NN}\|c_n\|<\pinf$, and 
let $(x_n)_{n\in\NN}$ and $(v_n)_{n\in\NN}$ be sequences generated 
by the following routine.
\begin{equation}
\label{e:main92}
\begin{array}{l}
\operatorname{Initialization}\\
\left\lfloor
\begin{array}{l}
\varepsilon\in\left]0,\min\{1,\|L\|^{-2}\}\right[\\[1mm]
v_0\in\GG\\[1mm]
\end{array}
\right.\\[5mm]
\operatorname{For}\;n=0,1,\ldots\\
\left\lfloor
\begin{array}{l}
y_n=z-L^*v_n\\[2mm]
\operatorname{if}\;\;d_C(y_n)>\max\partial\phi(0)\\
\left\lfloor
\begin{array}{l}
x_n=y_n+\displaystyle{\frac{\prox_{\phi^*}d_C(y_n)}{d_C(y_n)}}
(P_Cy_n-y_n)+b_n
\end{array}
\right.\\[4mm]
\operatorname{if}\;\;d_C(y_n)\leq\max\partial\phi(0)\\
\left\lfloor
\begin{array}{l}
x_n=P_Cy_n+b_n
\end{array}
\right.\\[2mm]
\gamma_n\in\left[\varepsilon,2\|L\|^{-2}-\varepsilon\right]\\[1mm]
w_n=\gamma_n^{-1}v_n+Lx_n-r\\[1mm]
\operatorname{if}\;\;d_D(w_n)>\gamma_n^{-1}\max\partial\psi(0)\\
\left\lfloor
\begin{array}{l}
p_n=\displaystyle{\frac{\prox_{(\gamma_n^{-1}\psi)^*}
d_D(w_n)}{d_D(w_n)}}(w_n-P_Dw_n)+c_n
\end{array}
\right.\\[4mm]
\operatorname{if}\;\;d_D(w_n)\leq\gamma_n^{-1}\max\partial\psi(0)\\
\left\lfloor
\begin{array}{l}
p_n=w_n-P_Dw_n+c_n
\end{array}
\right.\\[2mm]
\lambda_n\in\left[\varepsilon,1\right]\\
v_{n+1}=v_n+\lambda_n\big(\gamma_np_n-v_n\big).
\end{array}
\right.\\[2mm]
\end{array}
\end{equation}
Then the following hold, where $x$ designates the primal
solution to Problem~\ref{prob:92}.
\begin{enumerate}
\item
\label{p:92i}
$(v_n)_{n\in\NN}$ converges weakly to a solution $v$ to 
\eqref{e:prob44'} and, if we set $y=z-L^*v$,
\begin{equation}
\label{e:25fevrier2009}
x=
\begin{cases}
y+\displaystyle{\frac{\prox_{\phi^*}d_C(y)}{d_C(y)}}
(P_Cy-y),&\text{if}\;\;d_C(y)>\max\partial\phi(0);\\
P_Cy,&\text{if}\;\;d_C(y)\leq\max\partial\phi(0).
\end{cases}
\end{equation}
\item
\label{p:92iii}
$(x_n)_{n\in\NN}$ converges strongly to $x$.
\end{enumerate}
\end{proposition}
\begin{proof}
Set $f=\phi\circ d_C$ and $g=\psi\circ d_D$. Since $d_C$ and $d_D$ are
continuous convex functions, $f\in\Gamma_0(\HH)$ and 
$g\in\Gamma_0(\GG)$. Moreover, \eqref{e:1968} implies that
\eqref{e:1938} holds. Thus, Problem~\ref{prob:92} is a special case 
of Problem~\ref{prob:1}. On the other hand, it follows from 
Lemma~\ref{l:6}\ref{l:6i} that 
$\widetilde{f^*}=\|\cdot\|^2/2-(\phi\circ d_C)^\sim$ and from 
\cite[Lemma~2.2]{Luis09} that $g^*=\sigma_D+\psi^*\circ\|\cdot\|$.
This shows that \eqref{e:prob44'} is the dual of \eqref{e:prob44}.
Let us now examine iteration $n$ of the algorithm. In view of 
Example~\ref{ex:10}, the vector $x_n$ in \eqref{e:main92} is precisely 
the vector $x_n=\prox_f(z-L^*v_n)+b_n$ of \eqref{e:main1}. Moreover,
using successively the definition of $w_n$ in \eqref{e:main92}, 
Lemma~\ref{l:7}, Example~\ref{ex:10}, and the definition of $p_n$ in 
\eqref{e:main92}, we obtain
\begin{multline}
\gamma_n^{-1}\prox_{\gamma_n g^*}(v_n+\gamma_n(Lx_n-r))\\
\begin{aligned}[b]
&=\gamma_n^{-1}\prox_{\gamma_n g^*}(\gamma_n w_n)\\
&=w_n-\prox_{\gamma_n^{-1} g}w_n\\
&=w_n-\prox_{(\gamma_n^{-1}\psi)\circ d_D}w_n\\
&=
\begin{cases}
\displaystyle{\frac{\prox_{(\gamma_n^{-1}\psi)^*}
d_D(w_n)}{d_D(w_n)}}(w_n-P_Dw_n)&\text{if}\;\;d_D(w_n)>
\gamma_n^{-1}\max\partial\psi(0)\\
w_n-P_Dw_n&\text{if}\;\;d_D(w_n)\leq\gamma_n^{-1}\max\partial\psi(0)
\end{cases}
\\
&=p_n-c_n.
\end{aligned}
\end{multline}
Altogether, \eqref{e:main92} is a special instance of \eqref{e:main1} 
in which $(\forall n\in\NN)$ $a_n=\gamma_nc_n$. Therefore, since
$\sum_{n\in\NN}\|a_n\|\leq 2\|L\|^{-2}\sum_{n\in\NN}\|c_n\|<\pinf$, the
assertions follow from Theorem~\ref{t:1}, where we have
used \eqref{e:26juin2009} to get \eqref{e:25fevrier2009}.
\end{proof}

\begin{example}
\label{ex:hong-kong08-1}
We can obtain a soft-constrained version of the Potter-Arun problem
\eqref{e:1965} revisited in Example~\ref{ex:91} by 
specializing Problem~\ref{prob:92} as follows:
$z=0$, $\GG=\RR^N$, $D=\{0\}$, $r=(\rho_i)_{1\leq i\leq N}$,
and $L\colon x\mapsto(\scal{x}{s_i})_{1\leq i\leq N}$, where 
$(s_i)_{1\leq i\leq N}\in\HH^N$ satisfies 
$\sum_{i=1}^N\|s_i\|^2\leq1$. We thus arrive at the relaxed 
version of \eqref{e:1965}
\begin{equation}
\label{e:hong-kong2008-12}
\underset{x\in\HH}{\mathrm{minimize}}\;\;\phi(d_C(x))
+\psi\Big(\sqrt{\textstyle{\sum_{i=1}^N}
|\scal{x}{s_i}-\rho_i|^2}\Big)+\frac12\|x\|^2.
\end{equation}
Since $D=\{0\}$, we can replace each occurrence of $d_D(w_n)$ by 
$\|w_n\|$ and each occurrence of $w_n-P_Dw_n$ by $w_n$ in 
\eqref{e:main92}. Proposition~\ref{p:92}\ref{p:92iii} asserts that 
any sequence $(x_n)_{n\in\NN}$ produced by the resulting algorithm 
converges strongly to the solution to \eqref{e:hong-kong2008-12}. 
For the sake of illustration, let us consider the case when 
$\phi=\alpha|\cdot|^{4/3}$ and $\psi=\beta|\cdot|$, for some 
$\alpha$ and $\beta$ in $\RPP$. Then $\dom\psi=\RR$ and 
\eqref{e:1968} is trivially satisfied. In addition,
\eqref{e:hong-kong2008-12} becomes 
\begin{equation}
\label{e:hong-kong2008-13}
\underset{x\in\HH}{\mathrm{minimize}}\;\;\alpha d^{4/3}_C(x)
+\beta\sqrt{\textstyle{\sum_{i=1}^N}
|\scal{x}{s_i}-\rho_i|^2}+\frac12\|x\|^2.
\end{equation}
Since $\phi^*\colon\mu\mapsto 27|\mu|^4/(256\alpha^3)$,
$\prox_{\phi^*}$ in \eqref{e:main92} can be 
derived from Example~\ref{ex:phi}\ref{ex:phivi}. On the other hand,
since $\psi^*=\iota_{[-\beta,\beta]}$, 
Example~\ref{ex:1938}\ref{ex:1938i} yields
$\prox_{\psi^*}=P_{[-\beta,\beta]}$.
Thus, upon setting, for simplicity, $b_n\equiv 0$, $c_n\equiv 0$, 
$\lambda_n\equiv 1$, and $\gamma_n\equiv 1$ (note that $\|L\|\leq 1$) 
in \eqref{e:main92} and observing that $\partial\phi(0)=\{0\}$ 
and $\partial\psi(0)=[-\beta,\beta]$, we obtain the following algorithm, 
where $L^*\colon(\nu_i)_{1\leq i\leq N}\mapsto\sum_{i=1}^N\nu_is_i$.
\[
\begin{array}{l}
\operatorname{Initialization}\\
\left\lfloor
\begin{array}{l}
\tau={3}/(2\alpha 4^{1/3}),\;
\sigma={256\alpha^3}/{729}\\[3mm]
v_0\in\RR^N\\[1mm]
\end{array}
\right.\\[6mm]
\operatorname{For}\;n=0,1,\ldots\\
\left\lfloor
\begin{array}{l}
y_n=z-L^*v_n\\[2mm]
\operatorname{if}\;\;y_n\notin C\\
\left\lfloor
\begin{array}{l}
x_n=y_n+\displaystyle{\frac{
\bigg|{\sqrt{d^2_C(y_n)+\sigma}
+d_C(y_n)}\bigg|^{1/3}-\bigg|{\sqrt{d^2_C(y_n)+\sigma}-
d_C(y_n)}\bigg|^{1/3}}{\tau d_C(y_n)}}(P_Cy_n-y_n)\\[4mm]
\end{array}
\right.\\[4mm]
\operatorname{if}\;\;y_n\in C\\
\left\lfloor
\begin{array}{l}
x_n=y_n
\end{array}
\right.\\[2mm]
w_n=v_n+Lx_n-r\\[1mm]
\operatorname{if}\;\;\|w_n\|>\beta\\
\left\lfloor
\begin{array}{l}
v_{n+1}=\displaystyle{\frac{\beta}{\|w_n\|}}w_n
\end{array}
\right.\\[4mm]
\operatorname{if}\;\;\|w_n\|\leq\beta\\
\left\lfloor
\begin{array}{l}
v_{n+1}=w_n.
\end{array}
\right.\\[2mm]
\end{array}
\right.\\[2mm]
\end{array}
\]
As shown above, the sequence $(x_n)_{n\in\NN}$ converges strongly 
to the solution to \eqref{e:hong-kong2008-13}.
\end{example}

\begin{remark}
\label{ex:rio-juin2009}
Alternative relaxations of \eqref{e:1965} can be derived
from Problem~\ref{prob:1}. For instance, given an 
even function $\phi\in\Gamma_0(\RR)\smallsetminus\{\iota_{\{0\}}\}$ 
and $\alpha\in\RPP$, an alternative to \eqref{e:hong-kong2008-12} is 
\begin{equation}
\label{e:rio2009-05-20}
\underset{x\in\HH}{\mathrm{minimize}}\;\;
\phi(d_C(x))+\alpha\max_{1\leq i\leq N}|\scal{x}{s_i}-\rho_i|
+\frac12\|x\|^2.
\end{equation}
This formulation results from \eqref{e:prob1} with
$z=0$, $f=\phi\circ d_C$, $\GG=\RR^N$, $r=(\rho_i)_{1\leq i\leq N}$,
$L\colon x\mapsto(\scal{x}{s_i})_{1\leq i\leq N}$, and
$g=\alpha\|\cdot\|_{\infty}$ (note that 
\eqref{e:1938} holds since $\dom g=\GG$). Since $g^*=\iota_{D}$,
where $D=\menge{(\nu_i)_{1\leq i\leq N}\in\RR^N}
{\sum_{i=1}^N|\nu_i|\leq\alpha}$, the dual problem \eqref{e:prob2} 
therefore assumes the form 
\begin{equation}
\label{e:rio2009-2}
\underset{(\nu_i)_{1\leq i\leq N}\in D}{\mathrm{minimize}}\;\;
\frac12\bigg\|\sum_{i=1}^N\nu_is_i\bigg\|^2-{(\phi\circ d_C)^\sim}
\bigg(-\sum_{i=1}^N\nu_is_i\bigg)+\sum_{i=1}^N\rho_i\nu_i.
\end{equation}
The proximity operators of $f=\phi\circ d_C$ and 
$\gamma_n g^*=\iota_D$ required by Algorithm~\ref{algo:1} are supplied 
by Example~\ref{ex:10} and Example~\ref{ex:1938}\ref{ex:1938i},
respectively. Strong convergence of the resulting sequence 
$(x_n)_{n\in\NN}$ to the solution to \eqref{e:rio2009-05-20} is 
guaranteed by Theorem~\ref{t:1}\ref{t:1x}.
\end{remark}

\subsection{Denoising over dictionaries}
\label{sec:32}

In denoising problems, the goal is to recover the original form of 
an ideal signal $\overline{x}\in\HH$ from a corrupted observation 
\begin{equation}
\label{e:model}
z=\overline{x}+s,
\end{equation}
where $s\in\HH$ is the realization of a noise process which may for
instance model imperfections in the data recording instruments, 
uncontrolled dynamics, or physical interferences. 
A common approach to solve this problem is to minimize the 
least-squares data fitting functional $x\mapsto\|x-z\|^2/2$ subject 
to some constraints on $x$ that represent 
a priori knowledge on the ideal solution $\overline{x}$ and
some affine transformation $L\overline{x}-r$ thereof, where 
$L\in\BL(\HH,\GG)$ and $r\in\GG$. By measuring the degree of 
violation of these constraints via potentials $f\in\Gamma_0(\HH)$ and 
$g\in\Gamma_0(\GG)$, we arrive at \eqref{e:prob1}. In this context, 
$L$ can be a gradient \cite{Cham04,Dura07,Kark01,Rudi92}, a low-pass 
filter \cite{Andr77,Twom65}, a wavelet or a frame decomposition 
operator \cite{Jsts07,Dono94,Weav91}.  
Alternatively, the vector $r\in\GG$ may arise from the availability 
of a second observation in the form of a noise-corrupted 
linear measurement of $\overline{x}$, as in
\eqref{e:model2} \cite{Chau07}. 

In this section, the focus is placed on models in which information 
on the scalar products $(\scal{\overline{x}}{e_k})_{k\in\KK}$ of the 
original signal $\overline{x}$ against a finite or infinite a sequence 
of reference unit norm vectors $(e_k)_{k\in\KK}$ of $\HH$, called a 
dictionary, is available. In practice, such information can take 
various forms, e.g., sparsity, distribution type, statistical 
properties \cite{Chau07,Siop07,Daub04,Forn07,Mall99,Trop06}, and they 
can often be modeled in a variational framework by introducing a 
sequence of convex potentials $(\phi_k)_{k\in\KK}$. If we model the 
rest of the information available about $\overline{x}$ via a potential 
$f$, we obtain the following formulation.

\begin{problem}
\label{prob:11}
Let $z\in\HH$, let $f\in\Gamma_0(\HH)$, let $(e_k)_{k\in\KK}$ be a 
sequence of unit norm vectors in $\HH$ such that 
\begin{equation}
\label{e:frame1}
(\exi\delta\in\RPP)(\forall x\in\HH)\;\;\sum_{k\in\KK}|\scal{x}{e_k}|^2
\leq\delta\|x\|^2,
\end{equation}
and let $(\phi_k)_{k\in\KK}$ be functions in $\Gamma_0(\RR)$ such that 
\begin{equation}
\label{e:1937}
(\forall k\in\KK)\quad\phi_k\geq\phi_k(0)=0
\end{equation}
and 
\begin{equation}
\label{e:1939}
0\in\sri\Menge{\big(\scal{x}{e_k}-\xi_k\big)_{k\in\KK}}
{(\xi_k)_{k\in\KK}\in\ell^2(\KK),\;\;
\sum_{k\in\KK}\phi_k(\xi_k)<\pinf,\;\,\text{and}\;\,x\in\dom f}.
\end{equation}
The problem is to 
\begin{equation}
\label{e:prob11}
\underset{x\in\HH}{\mathrm{minimize}}\;\;
f(x)+\sum_{k\in\KK}\phi_k(\scal{x}{e_k})+\frac12\|x-z\|^2,
\end{equation}
and its dual is to
\begin{equation}
\label{e:prob21}
\underset{(\nu_k)_{k\in\KK}\in\ell^2(\KK)}{\mathrm{minimize}}\;\;
\widetilde{f^*}\bigg(z-\sum_{k\in\KK}\nu_{n,k}e_k\bigg)+
\sum_{k\in\KK}\phi^*_k(\nu_k).
\end{equation}
\end{problem}

Problems \eqref{e:prob11} and \eqref{e:prob21} can be solved by the 
following algorithm, where $\alpha_{n,k}$ stands for a 
numerical tolerance in the implementation of
the operator $\prox_{\gamma_n\phi_k^*}$. Let us note that closed-form
expressions for the proximity operators of a wide range of functions 
in $\Gamma_0(\RR)$ are available \cite{Chau07,Siop07,Smms05},
in particular in connection with Bayesian formulations involving
log-concave densities, and with problems involving sparse 
representations (see also Examples~\ref{ex:phi}--\ref{ex:8}
and Lemmas~\ref{l:santiago06-1}--\ref{l:3}).

\begin{proposition}
\label{p:11}
Let $((\alpha_{n,k})_{n\in\NN})_{k\in\KK}$ be sequences in $\RR$ such that
$\sum_{n\in\NN}\sqrt{\sum_{k\in\KK}|\alpha_{n,k}|^2}<\pinf$, let
$(b_{n})_{n\in\NN}$ be a sequence in $\HH$ such that 
$\sum_{n\in\NN}\|b_n\|<\pinf$, and let
$(x_n)_{n\in\NN}$ and $(v_n)_{n\in\NN}=((\nu_{n,k})_{k\in\KK})_{n\in\NN}$
be sequences generated by the following routine.
\begin{equation}
\label{e:main12}
\begin{array}{l}
\operatorname{Initialization}\\
\left\lfloor
\begin{array}{l}
\varepsilon\in\left]0,\min\{1,\delta^{-1}\}\right[\\[1mm]
(\nu_{0,k})_{k\in\KK}\in\ell^2(\KK)\\[1mm]
\end{array}
\right.\\[5mm]
\operatorname{For}\;n=0,1,\ldots\\
\left\lfloor
\begin{array}{l}
x_n=\prox_f\big(z-\sum_{k\in\KK}\nu_{n,k}e_k\big)+b_n\\[1mm]
\gamma_n\in\left[\varepsilon,2\delta^{-1}-\varepsilon\right]\\[1mm]
\lambda_n\in\left[\varepsilon,1\right]\\
\operatorname{For~every}\;k\in\KK\\
\left\lfloor
\begin{array}{l}
\nu_{n+1,k}=\nu_{n,k}+\lambda_n\big(\prox_{\gamma_n\phi_k^*}
(\nu_{n,k}+\gamma_n\scal{x_n}{e_k})+\alpha_{n,k}-\nu_{n,k}\big).
\end{array}
\right.\\[2mm]
\end{array}
\right.\\[2mm]
\end{array}
\end{equation}
Then the following hold, where $x$ designates the primal
solution to Problem~\ref{prob:11}.
\begin{enumerate}
\item
\label{p:11i}
$(v_n)_{n\in\NN}$ converges weakly to a solution 
$(\nu_k)_{k\in\KK}$ to \eqref{e:prob21} and 
$x=\prox_f(z-\sum_{k\in\KK}\nu_ke_k)$.
\item
\label{p:11ii}
$(x_n)_{n\in\NN}$ converges strongly to $x$.
\end{enumerate}
\end{proposition}
\begin{proof}
Set $\GG=\ell^2(\KK)$ and $r=0$. Define
\begin{equation}
\label{e:defF}
L\colon\HH\to\GG\colon x\mapsto(\scal{x}{e_k})_{k\in\KK}
\quad\text{and}\quad
g\colon\GG\to\RX\colon(\xi_k)_{k\in\KK}\mapsto
\sum_{k\in\KK}\phi_k(\xi_k). 
\end{equation}
Then $L\in\BL(\HH,\GG)$ and its adjoint is the operator 
$L^*\in\BL(\GG,\HH)$ defined by
\begin{equation}
L^*\colon(\xi_k)_{k\in\KK}\mapsto\sum_{k\in\KK}\xi_ke_k.
\end{equation}
On the other hand, it follows from our assumptions that
$g\in\Gamma_0(\GG)$ (Example~\ref{ex:chaux07}) and that
\begin{equation}
\label{e:1eroct2008}
g^*\colon\GG\to\RX\colon(\nu_k)_{k\in\KK}\mapsto
\sum_{k\in\KK}\phi^*_k(\nu_k).
\end{equation}
In addition, \eqref{e:1939} implies that \eqref{e:1938} holds.
This shows that \eqref{e:prob11} is a special case of 
\eqref{e:prob1} and that \eqref{e:prob21} is a special case of
\eqref{e:prob2}.
We also observe that \eqref{e:frame1} and \eqref{e:defF} yield
\begin{equation}
\|L\|^2=\sup_{\|x\|=1}\|Lx\|^2=
\sup_{\|x\|=1}\sum_{k\in\KK}|\scal{x}{e_k}|^2\leq\delta.
\end{equation}
Hence, $\left[\varepsilon,2\delta^{-1}-\varepsilon\right]\subset
\left[\varepsilon,2\|L\|^{-2}-\varepsilon\right] $. Next, we derive 
from \eqref{e:conjugate} and 
\eqref{e:1937} that, for every $k\in\KK$, 
$\phi_k^*(0)=\sup_{\xi\in\RR}-\phi_k(\xi)
=-\inf_{\xi\in\RR}\phi_k(\xi)=\phi_k(0)=0$ and that $(\forall\nu\in\RR)$
$\phi_k^*(\nu)=\sup_{\xi\in\RR}\xi\nu-\phi_k(\xi)\geq-\phi_k(0)=0$.
In turn, we derive from \eqref{e:1eroct2008} and 
Example~\ref{ex:chaux07} (applied to the canonical orthonormal basis of
$\ell^2(\KK)$) that
\begin{equation}
(\forall\gamma\in\RPP)(\forall v=(\nu_k)_{k\in\KK}\in\GG)\quad
\prox_{\gamma g^*}v=\big(\prox_{\gamma\phi_k^*}\nu_k\big)_{k\in\KK}.
\end{equation}
Altogether, \eqref{e:main12} is a special case of Algorithm~\ref{algo:1}
with $(\forall n\in\NN)$ $a_n=(\alpha_{n,k})_{k\in\KK}$. Hence,
the assertions follow from Theorem~\ref{t:1}.
\end{proof}

\begin{remark}
Using \eqref{e:defF}, we can write the potential on the dictionary 
coefficients in Problem~\ref{prob:11} as
\begin{equation}
\label{e:2009-07-01}
g\circ L\colon x\mapsto\sum_{k\in\KK}\phi_k(\scal{x}{e_k}).
\end{equation}
\begin{enumerate}
\item
If $(e_k)_{k\in\KK}$ were an orthonormal basis in Problem~\ref{prob:11},
we would have $L^{-1}=L^*$ and $\prox_{g\circ L}$ would be 
decomposable as $L^*\circ\prox_g\circ L$ \cite[Lemma~2.8]{Smms05}.
As seen in the Introduction, we could then approach \eqref{e:prob11}
directly via forward-backward, Douglas-Rachford, or Dykstra-like 
splitting, depending on the properties of $f$. Our duality framework 
allows us to solve \eqref{e:prob11} for the much broader class of 
dictionaries satisfying \eqref{e:frame1} and, in particular, for 
frames \cite{Daub92}. 
\item
Suppose that each $\phi_k$ in Problem~\ref{prob:11} is of the form
$\phi_k=\psi_k+\sigma_{\Omega_k}$, where $\psi_k\in\Gamma_0(\RR)$ 
satisfies $\psi_k\geq\psi_k(0)=0$ and is differentiable at $0$ with 
$\psi_k'(0)=0$, and where $\Omega_k$ is a nonempty closed interval. 
In this case, \eqref{e:2009-07-01} aims at promoting the sparsity of 
the solution in the dictionary $(e_k)_{k\in\KK}$ \cite{Siop07} (a 
standard case is when, for every $k\in\KK$, $\psi_k=0$ and 
$\Omega_k=[-\omega_k,\omega_k]$, which gives rise to the standard 
weighted $\ell^1$ potential 
$x\mapsto\sum_{k\in\KK}\omega_k|\scal{x}{e_k}|$). Moreover, the
proximity operator $\prox_{\gamma_n\phi_k^*}$ in \eqref{e:main12} can 
be evaluated via Lemma~\ref{l:7} and Lemma~\ref{l:santiago06-1}. 
\end{enumerate}
\end{remark}

\subsection{Denoising with support functions}
\label{sec:34}

Suppose that $g$ in Problem~\ref{prob:1} is positively homogeneous, i.e.,
\begin{equation}
\label{e:poshom}
(\forall\lambda\in\RPP)(\forall y\in\GG)\quad g(\lambda y)=\lambda g(y).
\end{equation}
Instances of such functions arising in denoising problems can be found 
in \cite{Amar94,Beck09,Bect04,Cham05,Siop07,Smms05,Daub07,%
Noll98,Rudi92,Weis09} and in the examples below.
It follows from \eqref{e:poshom} and \cite[Theorem~2.4.2]{Aubi90} that 
$g$ is the support function of a nonempty closed convex set 
$D\subset\GG$, namely
\begin{equation}
\label{e:jpa}
g=\sigma_D=\sup_{v\in D}\scal{\cdot}{v},\quad\text{where}\quad
D=\partial g(0)=
\menge{v\in\GG}{(\forall y\in\GG)\;\;\scal{y}{v}\leq g(y)}.
\end{equation}
If we denote by 
$\operatorname{bar}D=\menge{y\in\GG}{\sup_{v\in D}\scal{y}{v}<\pinf}$ 
the barrier cone of $D$, we thus obtain the following instance of 
Problem~\ref{prob:1}.

\begin{problem}
\label{prob:41}
Let $z\in\HH$, $r\in\GG$, let $f\in\Gamma_0(\HH)$, let $D$ be a 
nonempty closed convex subset of $\GG$, and let $L$ be a nonzero 
operator in $\BL(\HH,\GG)$ such that 
\begin{equation}
\label{e:1941}
r\in\sri\big(L(\dom f)-\operatorname{bar}D \big).
\end{equation}
The problem is to 
\begin{equation}
\label{e:prob41}
\underset{x\in\HH}{\mathrm{minimize}}\;\;
f(x)+\sigma_D(Lx-r)+\frac12\|x-z\|^2,
\end{equation}
and its dual is to
\begin{equation}
\label{e:prob42}
\underset{v\in D}{\mathrm{minimize}}\;\;
\widetilde{f^*}(z-L^*v)+\scal{v}{r}.
\end{equation}
\end{problem}

\begin{proposition}
\label{p:41}
Let $(a_{n})_{n\in\NN}$ be a sequence in $\GG$ such that
$\sum_{n\in\NN}\|a_n\|<\pinf$, let $(b_{n})_{n\in\NN}$ be a 
sequence in $\HH$ such that $\sum_{n\in\NN}\|b_n\|<\pinf$, and 
let $(x_n)_{n\in\NN}$ and $(v_n)_{n\in\NN}$ be sequences
generated by the following routine.
\begin{equation}
\label{e:main41}
\begin{array}{l}
\operatorname{Initialization}\\
\left\lfloor
\begin{array}{l}
\varepsilon\in\left]0,\min\{1,\|L\|^{-2}\}\right[\\[1mm]
v_0\in\GG\\[1mm]
\end{array}
\right.\\[5mm]
\operatorname{For}\;n=0,1,\ldots\\
\left\lfloor
\begin{array}{l}
x_n=\prox_f(z-L^*v_n)+b_n\\[1mm]
\gamma_n\in\left[\varepsilon,2\|L\|^{-2}-\varepsilon\right]\\[1mm]
\lambda_n\in\left[\varepsilon,1\right]\\
v_{n+1}=v_n+\lambda_n\big(P_D(v_n+\gamma_n(Lx_n-r))+a_n-v_n\big).
\end{array}
\right.\\[2mm]
\end{array}
\end{equation}
Then the following hold, where $x$ designates the primal
solution to Problem~\ref{prob:41}.
\begin{enumerate}
\item
\label{p:41i}
$(v_n)_{n\in\NN}$ converges weakly to a solution $v$ to 
\eqref{e:prob42} and $x=\prox_f(z-L^*v)$.
\item
\label{p:41i+}
$(x_n)_{n\in\NN}$ converges strongly to $x$.
\end{enumerate}
\end{proposition}
\begin{proof}
The assertions follow from Theorem~\ref{t:1} with $g=\sigma_D$. 
Indeed, $g^*=\iota_D$ and, therefore, $(\forall\gamma\in\RPP)$ 
$\prox_{\gamma g^*}=P_D$.
\end{proof}

\begin{remark}
\label{r:2009-01-07}
Condition \eqref{e:1941} is trivially satisfied when $D$ is 
bounded, in which case $\operatorname{bar}D=\GG$. 
\end{remark}

In the remainder of this section, we focus on examples that feature 
a bounded set $D$ onto which projections are easily computed. 

\begin{example}
\label{ex:41}
In Problem~\ref{prob:41}, let $D$ be the closed unit ball of $\GG$. Then 
$P_D\colon y\mapsto y/\max\{\|y\|,1\}$ and $\sigma_D=\|\cdot\|$. Hence,
\eqref{e:prob41} becomes
\begin{equation}
\label{e:prob411}
\underset{x\in\HH}{\mathrm{minimize}}\;\;
f(x)+\|Lx-r\|+\frac12\|x-z\|^2,
\end{equation}
and the dual problem \eqref{e:prob42} becomes
\begin{equation}
\label{e:prob412}
\underset{v\in\GG,\,\|v\|\leq 1}{\mathrm{minimize}}\;\;
\widetilde{f^*}(z-L^*v)+\scal{v}{r}.
\end{equation}
\end{example}

In signal recovery, variational formulations involving positively
homogeneous functionals to control the behavior of the gradient of 
the solutions play a prominent role, e.g., 
\cite{Aube06,Borw96,Huan08,Noll98,Rudi92}. 
In the context of image recovery, such a formulation can be obtained 
by revisiting Problem~\ref{prob:41} with $\HH=H_0^1(\Omega)$, where 
$\Omega$ is a bounded open domain in $\RR^2$, 
$\GG=L^2(\Omega)\oplus L^2(\Omega)$, $L=\nabla$, 
$D=\menge{y\in\GG}{|y|_2\leq\mu~\text{a.e.}}$ where $\mu\in\RPP$, 
and $r=0$. With this scenario, \eqref{e:prob41} is equivalent to 
\begin{equation}
\label{e:prob51}
\underset{x\in H_0^1(\Omega)}{\mathrm{minimize}}\;\;
f(x)+\mu\operatorname{tv}(x)+\frac12\|x-z\|^2,
\end{equation}
where $\operatorname{tv}(x)=\int_\Omega|\nabla x(\omega)|_{2}d\omega$.
In mechanics, such minimization problems have been studied extensively
for certain potentials $f$ \cite{Ekel99}. For instance, $f=0$ yields 
Mossolov's problem and its dual analysis is carried out in 
\cite[Section~IV.3.1]{Ekel99}. In image processing, Mossolov's problem
corresponds to the total variation denoising problem. Interestingly,
in 1980, Mercier \cite{Merc80} proposed a dual projection 
algorithm to solve Mossolov's problem. This approach was independently
rediscovered by Chambolle in a discrete setting \cite{Cham04,Cham05}.
Next, we apply our framework to a discrete version of \eqref{e:prob51} 
for $N\times N$ images. This will extend the method of \cite{Cham05}, 
which is restricted to $f=0$, and provide a formal proof for its 
convergence (see also \cite{Weis09} for an alternative scheme based on 
Nesterov's algorithm \cite{Nest05}). 

By way of preamble, let us introduce some notation. We denote by
$y=\big(\eta^{(1)}_{k,l},\eta^{(2)}_{k,l}\big)_{1\leq k,l\leq N}$ a 
generic element in $\RR^{N\times N}\oplus\RR^{N\times N}$ and by
\begin{equation}
\label{e:grad1}
\nabla\colon\RR^{N\times N}\to\RR^{N\times N}\oplus\RR^{N\times N}\colon
\big(\xi_{k,l}\big)_{1\leq k,l\leq N}\mapsto
\big(\eta^{(1)}_{k,l},\eta^{(2)}_{k,l}\big)_{1\leq k,l\leq N}
\end{equation}
the discrete gradient operator, where
\begin{equation}
\label{e:grad2}
(\forall (k,l)\in\{1,\ldots,N\}^2)\quad
\begin{cases}
\eta^{(1)}_{k,l}=\xi_{k+1,l}-\xi_{k,l},&\text{if}\;\;k<N;\\
\eta^{(1)}_{N,l}=0;\\
\eta^{(2)}_{k,l}=\xi_{k,l+1}-\xi_{k,l},&\text{if}\;\;l<N;\\
\eta^{(2)}_{k,N}=0.
\end{cases}
\end{equation}
Now let $p\in\left[1,\pinf\right]$. Then $p^*$ is the 
conjugate index of $p$, i.e., $p^*=\pinf$ if $p=1$, $p^*=1$ if 
$p=\pinf$, and $p^*=p/(p-1)$ otherwise. We define the $p$-th order
discrete total variation function as
\begin{equation}
\label{e:tv1}
\operatorname{tv}_p\colon\RR^{N\times N}\to\RR\colon x
\mapsto||\nabla x||_{p,1}\,,
\end{equation}
where 
\begin{equation}
\label{e:hue1}
(\forall y\in\RR^{N\times N}\oplus\RR^{N\times N})\quad\|y\|_{p,1}=
\sum_{1\leq k,l\leq N}\big|(\eta^{(1)}_{k,l},\eta^{(2)}_{k,l})\big|_p,
\end{equation}
with
\begin{equation}
\label{e:hue2}
\big(\forall(\eta^{(1)},\eta^{(2)})\in\RR^2\big)\quad
\big|(\eta^{(1)},\eta^{(2)})\big|_p=
\begin{cases}
\sqrt[p]{|\eta^{(1)}|^p+|\eta^{(2)}|^p},
&\text{if}\;\;p<\pinf;\\
\max\big\{|\eta^{(1)}|,|\eta^{(2)}|\big\},
&\text{if}\;\;p=\pinf.
\end{cases}
\end{equation}
In addition, the discrete divergence operator is defined as
\cite{Cham04} 
\begin{equation}
\label{e:div1}
\operatorname{div}\colon\RR^{N\times N}\oplus\RR^{N\times N}\to
\RR^{N\times N}\colon
\big(\eta^{(1)}_{k,l},\eta^{(2)}_{k,l}\big)_{1\leq k,l\leq N}\mapsto
\big(\xi^{(1)}_{k,l}+\xi^{(2)}_{k,l}\big)_{1\leq k,l\leq N},
\end{equation}
where
\begin{equation}
\label{e:div2}
\xi^{(1)}_{k,l}=
\begin{cases}
\eta^{(1)}_{1,l}&\text{if}\;\;k=1;\\
\eta^{(1)}_{k,l}-\eta^{(1)}_{k-1,l}&\text{if}\;\;1<k<N;\\
-\eta^{(1)}_{N-1,l}&\text{if}\;\;k=N;
\end{cases}
\quad\text{and}\quad
\xi^{(2)}_{k,l}=
\begin{cases}
\eta^{(2)}_{k,1}&\text{if}\;\;l=1;\\
\eta^{(2)}_{k,l}-\eta^{(2)}_{k,l-1}&\text{if}\;\;1<l<N;\\
-\eta^{(2)}_{k,N-1}&\text{if}\;\;l=N.
\end{cases}
\end{equation}

\begin{problem}
\label{prob:LLG}
Let $z\in\RR^{N\times N}$, let $f\in\Gamma_0(\RR^{N\times N})$, 
let $\mu\in\RPP$, let $p\in\left[1,\pinf\right]$, and set
\begin{equation}
\label{e:hue3}
D_p=\Menge{\big(\nu^{(1)}_{k,l},\nu^{(2)}_{k,l}\big)_{1\leq k,l\leq N}
\in\RR^{N\times N}\oplus\RR^{N\times N}}
{\max_{1\leq k,l\leq N}\big|(\nu^{(1)}_{k,l},
\nu^{(2)}_{k,l})\big|_{p^*}\leq 1}.
\end{equation}
The problem is to 
\begin{equation}
\label{e:probLLG}
\underset{x\in\RR^{N\times N}}{\mathrm{minimize}}\;\;
f(x)+\mu\operatorname{tv}_p(x)+\frac12\|x-z\|^2,
\end{equation}
and its dual is to
\begin{equation}
\label{e:probLLG'}
\underset{v\in D_p}
{\mathrm{minimize}}\;\;\widetilde{f^*}(z+\mu\operatorname{div}v).
\end{equation}
\end{problem}

\begin{proposition}
\label{p:LLG}
Let $\big(\alpha_{n,k,l}^{(1)}\big)_{n\in\NN}$ and 
$\big(\alpha_{n,k,l}^{(2)}\big)_{n\in\NN}$ be sequences 
in $\RR^{N\times N}$ such that 
\begin{equation}
\label{e:mainLLG2nde}
\sum_{n\in\NN}\sqrt{\sum_{1\leq k,l\leq N}
\big|\alpha_{n,k,l}^{(1)}\big|^2+\big|\alpha_{n,k,l}^{(2)}\big|^2}<\pinf,
\end{equation}
let $(b_{n})_{n\in\NN}$ be a 
sequence in $\RR^{N\times N}$ such that $\sum_{n\in\NN}\|b_n\|<\pinf$, 
and let $(x_n)_{n\in\NN}$ and $(v_n)_{n\in\NN}$ be sequences
generated by the following routine, where 
$(\pi^{(1)}_p{\mathsf y},\pi^{(2)}_p{\mathsf y})$ denotes the 
projection of a point ${\mathsf y}\in\RR^2$ onto the closed
unit $\ell^{p^*}$ ball in the Euclidean plane. 
\begin{equation}
\label{e:mainLLG}
\begin{array}{l}
\operatorname{Initialization}\\
\left\lfloor
\begin{array}{l}
\varepsilon\in\left]0,\min\{1,\mu^{-1}/8\}\right[\\[1mm]
v_0=\big(\nu^{(1)}_{0,k,l},\nu^{(2)}_{0,k,l}\big)_{1\leq k,l\leq N}
\in\RR^{N\times N}\oplus\RR^{N\times N}\\[1mm]
\end{array}
\right.\\[5mm]
\operatorname{For}\;n=0,1,\ldots\\
\left\lfloor
\begin{array}{l}
x_n=\prox_f(z+\mu\operatorname{div}v_n)+b_n\\[1mm]
\tau_n\in\left[\varepsilon,\mu^{-1}/4-\varepsilon\right]\\[1mm]
\big(\zeta^{(1)}_{n,k,l},\zeta^{(2)}_{n,k,l}\big)_{1\leq k,l\leq N}=
v_n+\tau_n\nabla x_n\\[2mm]
\lambda_n\in\left[\varepsilon,1\right]\\
\operatorname{For~every}\;(k,l)\in\{1,\dots,N\}^2\\[2mm]
\left\lfloor
\begin{array}{l}
\nu^{(1)}_{n+1,k,l}=\nu^{(1)}_{n,k,l}+\lambda_n\Big(
\pi_p^{(1)}\big(\zeta^{(1)}_{n,k,l},\zeta^{(2)}_{n,k,l}\big)
+\alpha^{(1)}_{n,k,l}-\nu^{(1)}_{n,k,l}\Big)\\[4mm]
\nu^{(2)}_{n+1,k,l}=\nu^{(2)}_{n,k,l}+\lambda_n\Big(
\pi_p^{(2)}\big(\zeta^{(1)}_{n,k,l},\zeta^{(2)}_{n,k,l}\big)
+\alpha^{(2)}_{n,k,l}-\nu^{(2)}_{n,k,l}\Big)\\[4mm]
\end{array}
\right.\\[8mm]
v_{n+1}=\big(\nu^{(1)}_{n+1,k,l},
\nu^{(2)}_{n+1,k,l}\big)_{1\leq k,l\leq N}\\
\end{array}
\right.
\end{array}
\end{equation}
Then $(v_n)_{n\in\NN}$ converges to a solution $v$ to 
\eqref{e:probLLG'}, $x=\prox_f(z+\mu\operatorname{div}v)$ is the 
primal solution to Problem~\ref{prob:LLG}, and $x_n\to x$.
\end{proposition}
\begin{proof}
It follows from \eqref{e:hue1} and \eqref{e:hue3} that
$\|\cdot\|_{p,1}=\sigma_{D_p}$. Hence, Problem~\ref{prob:LLG}
is a special case of Problem~\ref{prob:41} with $\HH=\RR^{N\times N}$,
$\GG=\RR^{N\times N}\oplus\RR^{N\times N}$, $L=\mu\nabla$ (see
\eqref{e:grad1}), $D=D_p$, and $r=0$. Moreover,
$L^*=-\mu\operatorname{div}$ (see \eqref{e:div1}),
$\|L\|=\mu\|\nabla\|\leq 2\sqrt{2}\mu$ \cite{Cham04}, and the 
projection of $y$ onto the set $D_p$ of \eqref{e:hue3} can be 
decomposed coordinatewise as
\begin{equation}
\label{e:PCp}
P_{D_p}y=\Big(\pi^{(1)}_p\big(\eta^{(1)}_{k,l},\eta^{(2)}_{k,l}\big),
\pi^{(2)}_p\big(\eta^{(1)}_{k,l},\eta^{(2)}_{k,l}\big)
\Big)_{1\leq k,l\leq N}.
\end{equation}
Altogether, upon setting, for every $n\in\NN$, $\tau_n=\mu\gamma_n$ and 
$a_n=\big(\alpha_{n,k,l}^{(1)},\alpha_{n,k,l}^{(2)}
\big)_{1\leq k,l\leq N}$, \eqref{e:mainLLG} appears as a special case of
\eqref{e:main41}. The results therefore follow from 
\eqref{e:mainLLG2nde} and Proposition~\ref{p:41}.
\end{proof}

\begin{remark}
\label{r:chamb}
The inner loop in \eqref{e:mainLLG} performs the projection step. For
certain values of $p$, this projection can be computed explicitly and we 
can therefore dispense with errors. Thus, if $p=1$, then $p^*=\pinf$ 
and the projection loop becomes
\begin{equation}
\label{e:main442}
\begin{array}{l}
\operatorname{For~every}\;(k,l)\in\{1,\dots,N\}^2\\
\left\lfloor
\begin{array}{l}
\nu^{(1)}_{n+1,k,l}=\nu^{(1)}_{n,k,l}+\lambda_n\Bigg(
\Frac{\zeta^{(1)}_{n,k,l}}{\max\big\{1,\big|\zeta^{(1)}_{n,k,l}
\big|\big\}}-\nu^{(1)}_{n,k,l}\Bigg)\\[1mm]
\nu^{(2)}_{n+1,k,l}=\nu^{(2)}_{n,k,l}+\lambda_n\Bigg(
\Frac{\zeta^{(2)}_{n,k,l}}
{\max\big\{1,\big|\zeta^{(2)}_{n,k,l}\big|\big\}}
-\nu^{(2)}_{n,k,l}\Bigg).
\end{array}
\right.\\[2mm]
\end{array}
\end{equation}
Likewise, if $p=2$, then $p^*=2$ and the projection loop becomes
\begin{equation}
\begin{array}{l}
\operatorname{For~every}\;(k,l)\in\{1,\dots,N\}^2\\[2mm]
\left\lfloor
\begin{array}{l}
\nu^{(1)}_{n+1,k,l}=\nu^{(1)}_{n,k,l}+\lambda_n\Bigg(
\Frac{\zeta^{(1)}_{n,k,l}}
{\max\big\{1,\big|\big(\zeta^{(1)}_{n,k,l},
\zeta^{(2)}_{n,k,l}\big)\big|_2\big\}}-\nu^{(1)}_{n,k,l}\Bigg)\\[5mm]
\nu^{(2)}_{n+1,k,l}=\nu^{(2)}_{n,k,l}+\lambda_n\Bigg(
\Frac{\zeta^{(2)}_{n,k,l}}
{\max\big\{1,\big|\big(\zeta^{(1)}_{n,k,l},
\zeta^{(2)}_{n,k,l}\big)\big|_2\big\}}-\nu^{(2)}_{n,k,l}\Bigg).
\end{array}
\right.\\[2mm]
\end{array}
\end{equation}
In the special case when $f=0$, $\lambda_n\equiv 1$, and 
$\tau_n\equiv\tau\in\left]0,\mu^{-1}/4\right[$
the two resulting algorithms reduce to the popular methods proposed
in \cite{Cham05}. Finally, if $p=\pinf$, then $p^*=1$ and the efficient
scheme described in \cite{Vand08} to project onto the $\ell^1$ ball 
can be used.
\end{remark}

\end{document}